\documentclass[12pt,letterpaper]{amsart}

\allowdisplaybreaks
\usepackage{mathpazo}
\usepackage{amsmath,amsfonts,amssymb}
\usepackage{amsrefs}
\usepackage{amsthm}
\usepackage{latexsym,amsmath,amssymb,amsfonts}
\usepackage{rotating}
\usepackage{mathrsfs}
\usepackage{xypic} \xyoption{all}
\usepackage{amscd}
\usepackage{hyperref} 
\usepackage{euscript}
\usepackage{hhline}
\usepackage{graphicx,epstopdf}
\usepackage{epsfig}
\usepackage{xcolor}
\usepackage{textcomp}
\usepackage[all,color]{xy}

\newlength{\fighskip} \fighskip=2pt
\newlength{\figvskip} \figvskip=3pt

\usepackage{hyperref}

\advance\textheight by \topskip
\numberwithin{equation}{section}

\newcommand{\C}{\mathbb{C}}

\newcommand{\R}{\mathbb{R}}

\newcommand{\g}{\mathfrak{g}}

\newcommand{\W}{\mathcal{W}}

\DeclareMathOperator{\End}{End}

\DeclareMathOperator{\Sym}{Sym}

\newcommand{\A}{\mathcal A}

\renewcommand{\W}{\mathcal W}

\theoremstyle{plain}
\newtheorem{thm}{Theorem}[section]
\newtheorem{thm-defn}{Theorem/Definition}[section]
\newtheorem{lem}[thm]{Lemma}
\newtheorem{lem-defn}[thm]{Lemma/Definition}
\newtheorem{prop}[thm]{Proposition}
\newtheorem{cor}[thm]{Corollary}

\theoremstyle{definition}
\newtheorem{defn}[thm]{Definition}
\newtheorem{notn}[thm]{Notation}
\newtheorem{eg}[thm]{Example}

\theoremstyle{remark}
\newtheorem{rmk}[thm]{Remark}

\allowdisplaybreaks[4]  

\begin{document}
\title[Quantizable functions on K\"ahler manifolds and non-formal quantization]{Quantizable functions on K\"ahler manifolds\\
	and non-formal quantization}

\author{Kwokwai Chan}
\address{Department of Mathematics, The Chinese University of Hong Kong, Shatin, Hong Kong}
\email{kwchan@math.cuhk.edu.hk}

\author{Naichung Conan Leung}
\address{The Institute of Mathematical Sciences and Department of Mathematics, The Chinese University of Hong Kong, Shatin, Hong Kong}
\email{leung@math.cuhk.edu.hk}

\author{Qin Li}
\address{Shenzhen Institute for Quantum Science and Engineering, Southern University of Science and Technology, Shenzhen, China}
\email{liqin@sustech.edu.cn}


\begin{abstract}
Applying the Fedosov connections constructed in \cite{CLL}, we find a (dense) subsheaf of smooth functions on a K\"ahler manifold $X$ which admits a non-formal deformation quantization. 
When $X$ is prequantizable and the Fedosov connection satisfies an integrality condition, we prove that this subsheaf of functions can be quantized to a sheaf of twisted differential operators (TDO), which is isomorphic to that associated to the prequantum line bundle. We also show that examples of such quantizable functions are given by images of quantum moment maps.


\end{abstract}

\maketitle

\tableofcontents

\section{Introduction}

The quantization of the phase space $(X,\omega)$ of a classical mechanical system is the procedure of associating functions in a dense subspace  $\mathcal{A}\subset C^\infty(X)$ to operators on a Hilbert space $\mathcal{H}$ such that the composition gives a deformation of the classical pointwise multiplication.  The most important two schemes of quantization in mathematics are deformation quantization and geometric quantization, which focus on different aspects of the quantization picture. Deformation quantization is by definition a formal deformation $(C^\infty(X)[[\hbar]],\star)$ of the commutative algebra $C^\infty(X)$, where $\hbar$ is a formal variable, such that the first term of the commutator is precisely the Poisson bracket; while geometric quantization focuses on the Hilbert space $\mathcal{H}$ and its operators.

In general, it is too optimistic to expect that the Hilbert space in geometric quantization is really a module over the deformation quantization algebra. We take the Berezin-Toeplitz quantization as an example (\cites{Ma-Ma-1,Ma-Ma, Bordemann-Meinrenken, Karabegov96, Karabegov00, Karabegov, Karabegov07}): a smooth function $f$ acts on the Hilbert spaces $\mathcal{H}_k := H^0(X,L^{\otimes k})$ of holomorphic sections of tensor powers of the prequantum line bundle via Toeplitz operators. However, the composition of Toeplitz operators is only \emph{asymptotic} to a sum of Toeplitz operators as $k\rightarrow \infty$. In particular, each Hilbert space $\mathcal{H}_k$ for a fixed $k$ does not form a module over the deformation quantization algebra. This is also one of the reasons for using the formal variable $\hbar$, instead of complex values, in Berezin-Toeplitz quantization. 

These drawbacks led to some dissatisfaction among physicists. They even claimed that ``deformation quantization is not quantization'' (see \cite[Section 1.4]{GW} for a more detailed explanation of this comment). To solve this problem, Gukov and Witten \cite{GW} proposed a new scheme of quantization by considering the $A$-model of a suitable complexification of a symplectic manifold $X$. In this \emph{brane quantization} picture, the Hilbert space and the algebra of operators acting on it are both morphism spaces between certain branes, whose definitions are still mysterious to both physicists and mathematicians. We also do not know in general which symplectic manifolds admit such a quantization (see, however, \cite{GW, Bischoff-Gualtieri}).


In this paper, we give a mathematical construction of \emph{non-formal} quantizations of K\"ahler manifolds. The naive idea is to take the evaluation of $\hbar$ in deformation quantization to some complex numbers to get rid of the formal variable. But there would be convergence issues in general. To overcome this, we exploit the Fedosov connections constructed in our previous work \cite{CLL}, which have nice finiteness properties. We also restrict to a subspace of smooth functions because the star product $f\star g$ of two functions is a formal power series in $\hbar$, which is in general divergent after evaluating $\hbar$ at complex values. We will be able to show that this subspace of functions can be quantized to holomorphic differential operators acting on the Hilbert space $\mathcal{H}_k = H^0(X,L^{\otimes k})$ in geometric quantization. This yields a dense subspace of \emph{quantizable functions} as $k\to \infty$.

To have a glimpse of the idea, consider the flat space $\C^n$ equipped with the Wick product on $C^\infty(\C^n)[[\hbar]]$, a natural choice of the subalgebra of $C^\infty(\C^n)$ is the space of polynomials where the formal variable can be evaluated at any complex number. This is because the Wick product of any two polynomials is still a polynomial in $\hbar$ (see equation \eqref{equation: fiberwise-Wick-product} for an explicit formula of Wick product).
On a general K\"ahler manifold $(X, \omega, J)$, we use Fedosov's flat connections \cite{Fed} to globalize the local computations and find a subspace of functions whose noncommutativity under the star product is polynomially controlled. 

The key lies in the fact that the Fedosov connections on K\"ahler manifolds constructed in \cite{CLL} are quantizations of Kapranov's $L_\infty$ structure \cite{Kapranov}. The special form of these Fedosov connections $D_\alpha$ allows us to take the evaluation $\hbar=1/k$ for any $k\in\C\setminus\{0\}$, yielding a \emph{non-formal} flat connection $D_{\alpha,k}$; here $\alpha$ (called the Karabegov form) is a $(1,1)$-form representing a class in $\hbar H^2_{dR}(X)[\hbar]$.
We call $k$ the \emph{level} and define \emph{quantizable functions of level $k$} as those functions whose corresponding flat sections under the connection $D_{\alpha,k}$ have only finite polynomial degree anti-holomorphic parts (see Definition \ref{definition: quantizable functions}). In particular, the star product of these functions is still quantizable and has only finite $\hbar$ power expansion, so there are no convergence issues. 
These quantizable functions form a sheaf $\mathcal{C}_{\alpha,k}^\infty$ of algebras on the K\"ahler manifold $X$ under the star product.

In Section \ref{section: TDO}, we show that this gives an example of so-called \emph{sheaves of twisted differential operators} (TDO for short) on $X$, which appeared in the theory of $D$-modules \cite{Ginzburg}:
\begin{thm}[= Theorems \ref{theorem: TDO} + \ref{theorem: characteristic class of TDO}]
	Let $X$ be a K\"ahler manifold. For any closed formal $(1,1)$-form $\alpha\in \hbar\A^{1,1}(X)[[\hbar]]$ and level $k$, the sheaf $\mathcal{C}_{\alpha,k}^\infty$ of quantizable functions (under the Fedosov connection $D_{\alpha,k}$) is a TDO on $X$ with characteristic class $[\omega - \alpha]$. 
\end{thm}

Note that the Karabegov form of the Fedosov connection $D_\alpha$ is precisely given by $\frac{1}{\hbar}(\omega - \alpha)$.
When it satisfies an integrality condition (see equation \eqref{equation: integrality condition}), we can prove that the sheaf $\mathcal{C}_{\alpha,k}^\infty$ is isomorphic to the sheaf of holomorphic differential operators on some holomorphic line bundles. 

A particularly important case is when the Karabegov form is the same as that in the \emph{Berezin-Toeplitz quantization} (when the closed formal $(1,1)$-form $\alpha$ is suitably chosen). In this case, the line bundles are tensor powers of the prequantum line bundle $L$. Applying the extension of Fedosov's method in our previous work \cite{CLL3} allows us to construct a \emph{level $k$ Bargmann-Fock sheaf} $\mathcal{F}_{L^{\otimes k}}$ of modules over the Weyl bundle $\W_{X,\C}$ equipped with a compatible Fedosov flat connection $D_{\alpha,k}$, for every positive integer $k$. Then we have the following result analogous to the ono-to-one correspondence \cite{Fed} between smooth functions and flat sections of the Weyl bundle in Fedosov quantization:
\begin{thm}[= Theorem \ref{theorem: Bargmann-Fock-isomorphic-prequantum}]
Suppose that $X$ is a K\"ahler manifold equipped with a prequantum line bundle $L$. Choose the closed formal $(1,1)$-form $\alpha$ so that the Karabegov form coincides with that of the Berezin-Toeplitz quantization of $X$.
Then for any positive integer $k$, the symbol map gives a canonical sheaf isomorphism from the sheaf of flat sections of the Bargmann-Fock sheaf $\mathcal{F}_{L^{\otimes k}}$ under its Fedosov connection $D_{\alpha,k}$ to the sheaf of holomorphic sections of the $k$-th tensor power $L^{\otimes k}$. 
\end{thm}

Now the compatibility between the Fedosov connections on the Weyl bundle $\W_{X,\C}$ and $\mathcal{F}_{L^{\otimes k}}$ implies that quantizable functions of level $k$ act on the (local) holomorphic sections of $L^{\otimes k}$. Since locality is obvious, this gives the desired non-formal quantization described by holomorphic differential operators on the Hilbert space $H^0(X,L^{\otimes k})$:
\begin{thm}[= Theorem \ref{theorem: almost-holomorphic-function-differential-operators}]\label{theorem: main-intro}
Suppose that $X$ is a K\"ahler manifold equipped with a prequantum line bundle $L$. Choose the closed formal $(1,1)$-form $\alpha$ so that the Karabegov form coincides with that of the Berezin-Toeplitz quantization of $X$.
Then for any positive integer $k$, there is a natural isomorphism (of TDOs)
$$
 \varphi: \mathcal{C}_{\alpha,k}^\infty\rightarrow \mathcal{D}(L^{\otimes k})
$$
from the sheaf of algebras of level $k$ quantizable functions to the sheaf of holomorphic differential operators on $L^{\otimes k}$. 
\end{thm}
This sheaf-theoretic description of our quantization procedure provides an example of gluing of quantizations over open sets to global ones in the K\"ahler setting. Theorem \ref{theorem: main-intro} can be generalized beyond the case of Berezin-Toeplitz quantization by changing the Karabegov form of the Fedosov quantization, so that holomorphic differential operators on \emph{any} holomorphic line bundle can be realized as quantizations of a class of quantizable functions. See the discussion at the end of Section \ref{section:quantizable functions as HDOs}.

Our notion of quantizable functions is a vast generalization of the previous notion of quantizable functions (or polarization-preserving functions) in geometric quantization (see e.g. \cite{Lerman}), which can only produce first order differential operators. In Section \ref{section: quantum-moment-maps}, we will see that quantizable functions in our sense also arise from Hamiltonian $G$-actions on the K\"ahler manifold $X$. More precisely, we will show that images of \emph{quantum moment maps} are all examples of first order quantizable functions in Theorem \ref{theorem: quantum-moment-maps-quantizable-functions}. From this, we obtain a Lie algebra homomorphism from the Lie algebra $\g$ of $G$ to the space of quantizable functions.
When $X$ is a flag variety $G/B$, this reproduces the Lie algebra representation in the Borel-Weil-Bott Theorem \cite{Bott}.

\subsection*{Conventions}
\begin{itemize}
	\item  Let $X$ be a smooth manifold. We denote by $\Omega^k_X$ the bundle of differential $k$-forms on $X$ and by $\Omega^\bullet_X = \bigoplus_{k}\Omega^k_X$ the full de Rham complex. Global smooth differential forms on  $X$ will be denoted by 
	$$
	\A_X^\bullet=\Gamma(X, \Omega^\bullet_X), \quad \text{where}\ \A_X^k=\Gamma(X, \Omega^k_X).
	$$
	Given a vector bundle $E$, the complex of $E$-valued differential forms is denoted by
	$$
	\A^\bullet_X(E)=\Gamma(X, \Omega^\bullet_X\otimes E). 
	$$
	\item For a complex manifold $X$, we let
	\begin{itemize}
		\item $TX$ and $T^*X$ denote the holomorphic tangent and cotangent bundles respectively;
		\item $\overline{TX}$ and $\overline{T^*X}$ denote the anti-holomorphic tangent and cotangent bundles respectively;
		\item $TX_\R$ and $T^*X_\R$ denote the real tangent and cotangent bundles respectively;
		\item $TX_\C$ and $T^*X_\C$ denote the complexified tangent and cotangent bundles respectively.
	\end{itemize}
\medskip

\item We use the Einstein summation convention throughout this paper. 
	\end{itemize}




\subsection*{Acknowledgement}

The authors thank Nikolas Ziming Ma and Shilin Yu for very helpful discussions on this project. We also thank the referees for many valuable comments and suggestions for improvement.
Kwokwai Chan was supported by a grant from the Hong Kong Research Grants Council (Project No. CUHK14301621) and direct grants from CUHK.
Naichung Conan Leung was supported by grants of the Hong Kong Research Grants Council (Project No. CUHK14301619 \& CUHK14301721) and a direct grant (Project No. 4053400) from CUHK.
Qin Li was supported by grants from National Science Foundation of China (Project No. 12071204) and Guangdong Basic and Applied Basic Research Foundation (Project No. 2020A1515011220). Qin Li also thanks the SUSTech International Center for Mathematics for hospitality. 

\section{Quantizable functions via Fedosov quantization}

Recall that a {\em deformation quantization} of a symplectic manifold $(X, \omega)$ is a formal deformation of the commutative algebra $(C^{\infty }(X),\cdot)$ equipped with pointwise multiplication to a noncommutative one $( C^{\infty }( X)[[\hbar]] ,\star) $ equipped with a {\em star product} of the following form
$$
f\star g=fg+\sum_{i\geq 1}\hbar^i\cdot C_i(f,g),
$$
where each $C_i(-,-)$ is a bi-differential operator, so that the leading order of noncommutativity is a constant multiple of the Poisson bracket $\{-,-\}$ associated to $\omega$, i.e.,
\begin{equation}\label{equation: Poisson-bracket}
C_1(f,g)-C_1(g,f)=\frac{d}{d\hbar}\Big|_{\hbar=0}\left( f\star_{\hbar}g - g\star_{\hbar}f\right)
=\frac{\sqrt{-1}}{2}\left\{ f,g\right\}.
\end{equation}
In \cite{Fed}, Fedosov gave a beautiful geometric construction of deformation quantizations on symplectic manifolds. 

\subsection{Fedosov quantization of a K\"ahler manifold}
\

In this section, we briefly review our construction of Fedosov quantization in the K\"ahler case in \cite{CLL}. We will focus on Wick type star products on K\"ahler manifolds. The K\"ahler form on a K\"ahler manifold $X$ will always be written in local coordinates as
$$
\omega=\omega_{i\bar{j}}dz^i\wedge d\bar{z}^j,
$$
where we adopt the convention that $\omega^{\bar{k}i}\omega_{i\bar{j}}=\delta_{\bar{j}}^{\bar{k}}$.

We consider the following \emph{Weyl bundles} on $X$:
\begin{align*}\label{equation: Weyl-bundle}
 \W_{X}& := \widehat{\Sym}T^*X, \quad \overline{\W}_X:=\widehat{\Sym}\overline{T^*X},\\
 \W_{X,\C}& := \W_{X}\otimes_{\mathcal{C}^\infty_X}\overline{\W}_X=\widehat{\Sym}T^*X_{\C}.
\end{align*}
To give explicit expressions of these bundles, we let $(z^1,\cdots, z^n)$ be a local holomorphic coordinate system on $X$, use $dz^i,d\bar{z}^j$'s to denote $1$-forms in $\A_X^\bullet$ and use $y^i,\bar{y}^j$ to denote sections in $\W_{X,\C}$. The K\"ahler form enables us to define a non-commutative fiberwise Wick product on $\W_{X,\C}$: 
\begin{equation}\label{equation: fiberwise-Wick-product}
 a\star b := \sum_{k\geq 0}\frac{\hbar^k}{k!}\cdot\omega^{i_1\bar{j}_1}\cdots\omega^{i_k\bar{j}_k}\cdot\frac{\partial^k a}{\partial y^{i_1}\cdots\partial y^{i_k}}\frac{\partial^k b}{\partial \bar{y}^{j_1}\cdots\partial \bar{y}^{j_k}}.
\end{equation}

Throughout this paper, we denote by $\nabla$ the Levi-Civita connection on $X$, and its natural extension to the Weyl bundle $\W_{X,\C}$. By \cite[Proposition 4.1]{Bordemann}, its curvature can be written as a bracket:
$$
\nabla^2=\frac{1}{\hbar}[R_\nabla,-]_\star, 
$$
where $R_\nabla=R_{i\bar{j}k\bar{l}}dz^i\wedge d\bar{z}^j\otimes y^k\bar{y}^l\in\A_X^2(\W_{X,\C})$. 

A natural filtration on these Weyl bundles is defined by polynomial degrees. For instance, $(\overline{\W_X})_{\leq N}$ denotes the sum of anti-holomorphic monomials of polynomial degree $\leq N$. 
The {\em symbol map}
\begin{equation}\label{equation: symbol-map}
	\sigma: \A_X^\bullet(\W_{X,\C})[[\hbar]]\rightarrow\A_X^\bullet[[\hbar]].
\end{equation}
is defined by setting all $y^i,\bar{y}^j$'s to zero. Here $\A_X^\bullet(\W_{X,\C})[[\hbar]]$ denotes the complex of differential forms on $X$ with values in the Weyl bundle. 

\begin{defn}\label{definition: operators}
We will use the notation $\mathcal{W}_{p,q}$ to denote the component $\Sym^p T^*X\otimes_{\mathcal{C}^\infty_X}\Sym^q\overline{T^*X}$ of $\mathcal{W}_{X,\mathbb{C}}$; sections of this subbundle are said to be \emph{of type $(p,q)$}. There are four natural operators acting as derivations on $\A_X^\bullet(\mathcal{W}_{X,\mathbb{C}})$:
\begin{align*}
\delta^{1,0} a  = dz^i\wedge\frac{\partial a}{\partial y^i},\quad 
\delta^{0,1}a  = d\bar{z}^j\wedge\frac{\partial a}{\partial\bar{y}^j},
\end{align*}
as well as
\begin{align*}
(\delta^{1,0})^*a  = y^k\cdot \iota_{\partial_{z^k}}a, \quad
(\delta^{0,1})^*a  = \bar{y}^j\cdot \iota_{\partial_{\bar{z}^j}}a.
\end{align*}
We define the operators $(\delta^{1,0})^{-1}$ and $(\delta^{0,1})^{-1}$ by normalizing $(\delta^{1,0})^{*}$ and $(\delta^{1,0})^{*}$ respectively: 
\begin{equation*}\label{equation: delta-1-0-inverse}
 (\delta^{1,0})^{-1}:=\frac{1}{p_1+p_2}(\delta^{1,0})^*\ \text{on $\A_X^{p_1,q_1}(\mathcal{W}_{p_2,q_2})$},
\end{equation*}
\begin{equation*}\label{equation: delta-0-1-inverse}
 (\delta^{0,1})^{-1}:=\frac{1}{q_1+q_2}(\delta^{0,1})^*\ \text{on $\A_X^{p_1,q_1}(\mathcal{W}_{p_2,q_2})$}.
\end{equation*}
\end{defn}
\begin{rmk}
	It is not difficult to see that the operators in Definition \ref{definition: operators} are all independent of the coordinates chosen. 
\end{rmk}
Let $\pi_{0,*}$ be the natural projection from $\A_X^\bullet(\mathcal{W}_{X,\mathbb{C}})$ to $\A_X^{0,\bullet}(\overline{\mathcal{W}}_X)$. Then
we have the following useful equality:
\begin{equation}\label{equation: delta-1-0-and-inverse}
\text{id}-\pi_{0,*}  = \delta^{1,0}\circ(\delta^{1,0})^{-1}+(\delta^{1,0})^{-1}\circ\delta^{1,0}.
\end{equation}
We also define the fiberwise de Rham differential as $\delta:=\delta^{1,0}+\delta^{0,1}$. 
\begin{defn}
A connection on the formal Weyl bundle $\W_{X,\C}[[\hbar]]$ of the form 
$$
D = \nabla-\delta+\frac{1}{\hbar}[I,-]_{\star}
$$
is called a {\em Fedosov connection} if $D$ is flat, i.e. $D^2=0$.
Here $\nabla$ is the Levi-Civita connection,  and $I\in\A^1_X(\W_{X,\C})[[\hbar]]$ is a $1$-form valued section of $\W_{X,\C}[[\hbar]]$. Such a connection can be extended to a differential on $\A^\bullet_X(\W_{X,\C})$.
\end{defn}

\begin{notn}
	Let $\nabla$ be the Levi-Civita connection. We define the following operator 
	$$
	\tilde{\nabla}^{1,0}:=(\delta^{1,0})^{-1}\circ\nabla^{1,0};
	$$
	the operator $\tilde{\nabla}^{0,1}$ is similarly defined. 
\end{notn}
For later computations, we need the following
\begin{lem}\label{lemma: flat-section-1-0-comonent}
For any $k\geq 0$, given any $\alpha\in\Gamma(X,\Sym^k\overline{TX})$, there exists a unique $\tilde{\alpha}$ such that $(\nabla^{1,0}-\delta^{1,0})(\tilde{\alpha})=0$ and that $\pi_{0,*}(\tilde{\alpha})=\alpha$. 
\end{lem}
\begin{proof}
For $k\geq 0$, let $\alpha_k := (\tilde{\nabla}^{1,0})^k(\alpha)$, and define $\tilde{\alpha}$ by
$$
\tilde{\alpha} := \sum_{k\geq 0}\alpha_k. 
$$
Clearly, $\pi_{0,*}(\tilde{\alpha})=\alpha$. According to our construction, we then have $\alpha_{k+1}=(\delta^{1,0})^{-1}\left(\nabla^{1,0}\alpha_k\right)$
because $\delta^{1,0}(\nabla^{1,0}\alpha_k)=0$ for all $k\geq 0$. Applying equation \eqref{equation: delta-1-0-and-inverse}, we obtain 
\begin{align*}
\delta^{1,0}\alpha_{k+1} & = \delta^{1,0}\circ(\delta^{1,0})^{-1}\left(\nabla^{1,0}\alpha_k\right)\\
& = \delta^{1,0}\circ(\delta^{1,0})^{-1}\left(\nabla^{1,0}\alpha_k\right)+(\delta^{1,0})^{-1}\circ\delta^{1,0}\left(\nabla^{1,0}\alpha_k\right) = \nabla^{1,0}\alpha_k.
\end{align*}
Thus we get
\begin{align*}
	(\nabla^{1,0}-\delta^{1,0})(\tilde{\alpha}) & = -\delta^{1,0}(\alpha_0)+\nabla^{1,0}(\alpha_0)-\delta^{1,0}(\alpha_1)+\nabla^{1,0}(\alpha_1)-\delta^{1,0}(\alpha_2)\cdots\\
	& = -\delta^{1,0}(\alpha_0) = 0. 
\end{align*}
\end{proof}

The main result in Fedosov's approach to deformation quantization is summarized in the following 
\begin{thm}[Fedosov \cite{Fed}]\label{theorem: original Fedosov}
	There exist Fedosov connections on the Weyl bundle $\W_{X,\C}[[\hbar]]$. Furthermore, for every formal smooth function $f\in C^\infty(X)[[\hbar]]$, there is a unique flat section $O_f$ of the Weyl bundle with $\sigma(O_f)=f$. The associated deformation quantization (or star product) is defined by the formula 
	$$
	O_f\star O_g = O_{f\star g}.
	$$
\end{thm}


In \cite{CLL}, we showed that a class of Fedosov connections can be obtained by quantizing Kapranov's $L_\infty$ structure on a K\"ahler manifold \cite{Kapranov}:
\begin{thm}[Theorems 2.17 and 2.25 in \cite{CLL}]\label{theorem: Fedosov-connection}
Let $\alpha = \sum_{i\geq 1}\hbar^i\alpha_i$ be a representative of a formal cohomology class in $\hbar H^2_{dR}(X)[[\hbar]]$ of type $(1,1)$. Then there exists a solution of the form $I_\alpha = I+ J_\alpha \in \mathcal{A}_X^{0,1}(\mathcal{W}_{X,\mathbb{C}})$ of the \emph{Fedosov equation}, namely,
\begin{equation}\label{equation: Fedosov-equation}
\nabla I_\alpha - \delta I_\alpha + \frac{1}{\hbar} I_\alpha\star I_\alpha + R_\nabla=\alpha.
\end{equation}
We denote the corresponding Fedosov connection by
$$D_{\alpha} := \nabla-\delta+\frac{1}{\hbar}[I_\alpha, -]_{\star}.$$
The deformation quantization associated to the flat connection $D_{\alpha}$ is a Wick type star product with Karabegov form given by $\frac{1}{\hbar}(\omega-\alpha)$. 
\end{thm}

Let us explain the notations in the decomposition $I_\alpha = I+ J_\alpha$ in Theorem \ref{theorem: Fedosov-connection} (see \cite[Section 2.3]{CLL} for more details):
The term $I$ in the connection $D_\alpha$ is obtained by repeatedly taking covariant exterior derivatives to the curvature tensor. 
Note that the connection $D = D_0 = \nabla-\delta+\frac{1}{\hbar}[I, -]_{\star}$ is also flat and $I$ satisfies the Fedosov equation \eqref{equation: Fedosov-equation} with $\alpha = 0$.
A key property is that $I$ is \emph{uniformly of polynomial degree $1$ in $\overline{\W}_X$}, and we can  decompose $I=\sum_{k\geq 2}I_k$ according to the polynomial degrees in the holomorphic Weyl bundle $\W_X$. In particular, the first term is given by 
$$
I_2=(\delta^{1,0})^{-1}\left(R_{i\bar{j}k{l}}dz^i\wedge d\bar{z}^j\otimes y^k\bar{y}^l\right).
$$
The term $J_\alpha$ is given as follows. Let $\varphi$ be a (locally defined) function such that $\partial\bar{\partial}\varphi=\alpha$. Then we set $J_\alpha:=-\sum_{k\geq 1}(\tilde{\nabla}^{1,0})^k(\bar{\partial}\varphi)$. Thus $J_\alpha$ is independent of the choice of $\varphi$ and uniformly of polynomial degree $0$ in $\overline{\W}_X$.

\begin{rmk}\label{remark: delta}
The operator $\delta$ can also be written as a bracket. Accordingly, the Fedosov connection $D_\alpha$ can be written as 
$$
D_{\alpha}=\nabla+\frac{1}{\hbar}[\gamma_\alpha,-]_\star,
$$
where $\gamma_\alpha=\omega_{i\bar{j}}(d\bar{z}^j\otimes y^i-dz^i\otimes\bar{y}^j)+I_\alpha$. The flatness of $D_{\alpha}$ is then equivalent to the following version of the Fedosov equation:
\begin{equation}\label{equation: Fedosov-equation-gamma}
 \nabla\gamma_\alpha+\frac{1}{\hbar}\gamma_\alpha\star\gamma_\alpha+R_\nabla=-\omega+\alpha.
\end{equation}
\end{rmk}

If the formal $(1,1)$-form $\alpha$ in the Fedosov equation \eqref{equation: Fedosov-equation} is a polynomial in $\hbar$, then we call the Fedosov connection $D_\alpha$ and the associated star product {\em admissible}. In this paper, we will only consider admissible Fedosov connections.
Comparing to several previous Fedosov constructions of Wick type star products \cites{Bordemann, Karabegov00, Neumaier}, there are several nice properties of our Fedosov connections $D_{\alpha}$, which will play important roles in this paper:
\begin{itemize}
	\item First of all, the Karabegov form of the associated star product can be read off from the Fedosov equation \eqref{equation: Fedosov-equation-gamma}. 
	\item Secondly, if the formal $(1,1)$-form $\alpha$ is only a polynomial in $\hbar$, then the term $I_\alpha$ in the Fedosov connection $D_\alpha$ is also a polynomial in $\hbar$. This enables us to evaluate $\hbar$ at any complex number without convergence issues. This only works for our construction of Fedosov connections on K\"ahler manifolds and is significantly different from Fedosov's original construction. 
	\item Lastly, it was shown in \cite{CLL} that for a (local) holomorphic function $f$, its associated flat section $O_f$ is only a section of the holomorphic Weyl bundle $\W_X$. (This fact is independent of the closed formal $(1,1)$-form $\alpha$.) 
\end{itemize}

\subsection{Quantizable functions}
\

To define quantizable functions, we need the following
\begin{defn}
	We define a \emph{weight} on $\W_{X,\C}[[\hbar]]$ by assigning weights on its generators:
	\begin{equation}\label{equation: weights-formal-Weyl-bundle}
		|y^i|=0, \hspace{2mm} |\bar{y}^j|=2, \hspace{2mm}|\hbar|=2.
	\end{equation}
   This weight is compatible with the fiberwise Wick product $\star$, in the sense that the product preserves the weight. 
   It is clear that a section of $\W_{X,\C}$ is \emph{of finite weight} if and only if it is both a polynomial in $\hbar$ and $\bar{y}^j$'s. There is an associated increasing filtration on the Weyl bundle; explicitly, we let $(\W_{X,\C}[[\hbar]])_N$ denote sums of monomials with weights $\leq N$. 
\end{defn}
\begin{rmk}
 A section of the formal Weyl bundle lives in a finite filtration component if and only if it lives in $\Sym^\bullet\overline{T^*X}\otimes\W_X[\hbar]$.
\end{rmk}
\begin{rmk}
	This weight is different from the one in \cite{Fed}, although both are compatible with the fiberwise Wick product. The weight we just defined is a \emph{polarized} version, namely, only anti-holomorphic terms in $\W_{X,\C}$ have non-zero weights. 
\end{rmk}
Admissible Fedosov connections $D_{\alpha}$ of polynomial degree 1 in $\hbar$ have the following nice property:
\begin{lem}
	Suppose $D_\alpha$ is an admissible Fedosov connection of polynomial degree $l$ in $\hbar$. Then for any $N\geq 0$, we have $D_\alpha\left((\W_{X,\C}[[\hbar]])_N\right)\subset (\W_{X,\C}[[\hbar]])_{N+2l}$.
\end{lem}

\begin{defn}\label{definition: formal-quantizable-functions}
A {\em formal quantizable function} is a formal function $f\in C^\infty(X)[[\hbar]]$ whose associated flat section $O_f$ lives in a finite filtration component of $\W_{X,\C}[[\hbar]]$,  or equivalently,  $O_f\in\Sym^\bullet\overline{T^*X}\otimes\W_X[\hbar]$. These functions can also be defined on any open subset of $X$, so they define a sheaf which we denote by $\mathcal{C}^\infty_{q,\hbar}$. (The subscript $q$ here stands for ``quantizable''.)
\end{defn}

\begin{eg}\label{example: holomorphic-functions}
	Every (local) holomorphic function $f$ is a formal quantizable function for \emph{any} Karabegov form. Explicitly, the flat section associated to $f$ is given by 
	$$
	O_f=\sum_{k\geq 0}(\tilde{\nabla}^{1,0})^k(f),
	$$
	We have seen in Lemma \ref{lemm: flat-section-1-0-components} that $D_{\alpha}^{1,0}(O_f)=0$. On the other hand, we have
	$$
	D_{\alpha}^{0,1}(O_f)=\sum_{k\geq 0}(\tilde{\nabla}^{1,0})^k(\bar{\partial}f)=0,
	$$
	since $f$ is holomorphic. Thus  $O_f\in\left(\W_{X,C}[[\hbar]]\right)_0=\W_X$, making $f$ a formal quantizable function. 
\end{eg}
A natural question is whether there are examples of formal quantizable functions other than holomorphic ones. We will answer this question by giving an explicit class of such formal functions in the following proposition, which will play an important role in later sections.

\begin{prop}\label{proposition: function-u-k}
	Let $\alpha=\sum_{i\geq 1}\hbar^i\alpha_i$ be a formal closed differential form of type $(1,1)$ and $\sum_{i\geq 1}\hbar^i\rho_i$ be a potential of $\alpha$ (i.e., $\partial\bar{\partial}\rho_i=\alpha_i$ for each $i$), and let $\rho$ be a potential of $\omega$. Then the (locally defined) formal functions $ u_k=\frac{\partial}{\partial z^k}\left(\rho - \sum_{i\geq 1}\hbar^i \rho_i\right)$ satisfy the following two properties:
	\begin{enumerate}
		\item The anti-holomorphic terms in $O_{u_k}$ have polynomial degrees at most $1$, i.e., $O_{u_k}$ is a section of $\W_X\otimes(\overline{\W_X})_{\leq 1}$.
		\item The terms in $O_{u_k}$ which live in $\overline{\mathcal{W}}_X$ (which we call ``terms of purely anti-holomorphic type'') are given by  
		\begin{equation}\label{equation: local-quantizable-function-anti-holomorphic}
		u_k+\omega_{k\bar{m}}\bar{y}^m.
		\end{equation}
	\end{enumerate}
	Hence if the Fedosov connection $D_\alpha$ is admissible (i.e., $\alpha$ is a polynomial in $\hbar$), then these $u_k$'s are all formal quantizable functions. 
\end{prop}

\begin{proof}
	For simplicity, we will prove the case where $\alpha$ has only one term, namely, $\alpha=\hbar\alpha_1$; the general case can be proven similarly. 	Recall that  $O_{u_k}$ is the unique solution of the iterative equation:
	\begin{align*}
		O_{u_k} = u_k+\delta^{-1}\circ\left(\nabla+\frac{1}{\hbar}[I_\alpha,-]_{\star}\right)(O_{u_k}).
	\end{align*}
	Observe that if a monomial $A$ does not live in $\A_X^{\bullet}(\overline{\mathcal{W}}_{X})$, then $\nabla A+\frac{1}{\hbar}[I_\alpha,A]_{\star}$ does not have terms living in $\A_X^{\bullet}(\overline{\mathcal{W}}_{X})$.  So we can prove the theorem by an induction on the weights of ``terms of purely anti-holomorphic type'' in $O_{u_k}$.
	
	The terms in $O_{u_k}$ of weight $1$ are given by
	$$
	\frac{\partial^2\rho}{\partial\bar{z}^l\partial z^k}\bar{y}^l=\omega_{k\bar{l}}\bar{y}^l.
	$$
	We know from the above iterative equation that the weight $2$ terms are given by
	$$
	\delta^{-1}\circ\nabla^{0,1}(\omega_{k\bar{l}}\bar{y}^l),
	$$
	which vanish since the Levi-Civita connection is compatible with both the symplectic form and the complex structure. The next terms are 
	\begin{align*}
		  & \delta^{-1}\left(\nabla^{0,1}\left(-\hbar\frac{\partial\rho_1}{\partial z^k}\right)+\frac{1}{\hbar}\left[-\hbar\frac{\partial^2\rho_1}{\partial\bar{z}^n\partial z^m}d\bar{z}^n\otimes y^m,\omega_{k\bar{l}}\bar{y}^l\right]_{\star}\right)\\
		= & \delta^{-1}\left(-\hbar\frac{\partial^2\rho_1}{\partial z^k\partial\bar{z}^l}d\bar{z}^l-\hbar\frac{\partial^2\rho_1}{\partial\bar{z}^n\partial z^m}d\bar{z}^n\omega_{k\bar{l}}\omega^{m\bar{l}}\right)=0.
	\end{align*}
	Thus the weight $3$ terms of purely anti-holomorphic type in $O_{u_k}$ vanish. This argument can be generalized to all such terms of higher weights. 
\end{proof}

\begin{rmk}
These formal functions are generalizations of the holomorphic partial derivatives of the K\"ahler potential $\rho=\sum_i z^i\bar{z}^i$ on the flat space $\C^n$: 
$$f_i=\frac{\partial\rho}{\partial z^i}=\bar{z}^i,$$
which are also quantizable functions.  
\end{rmk}



Given an admissible Fedosov connection $D_\alpha$, as it is only a polynomial in $\hbar$ instead of a formal power series, we can evaluate $D_{\alpha}$ at $\hbar=1/k$ for any non-zero complex number $k\in\C\setminus\{0\}$ and obtain the following \emph{non-formal} flat connection:
$$
D_{\alpha,k}=\nabla-\delta+k\cdot[I_{\alpha,k},-]_{\star_k},
$$
which acts as a differential on the Weyl bundle $\A^\bullet_X\left(\W_{X,\C}\right)$. 
\begin{rmk}
	Although this connection is non-formal (meaning that there is no $\hbar$ involved), we will still call it a Fedosov connection by abuse of notation. 
\end{rmk}

\begin{rmk}
	The fact that the differential $D_{\alpha,k}$ is well-defined relies heavily on the property that the Fedosov connections $D_\alpha$ are quantizations of Kapranov's $L_\infty$ structure \cite{Kapranov}. General Fedosov connections cannot be evaluated at arbitrary non-zero complex values because they are power series in $\hbar$ and there is no convergence in general. 
\end{rmk}

\begin{rmk}
	In general, we add a subscript $k$ in a symbol to denote the evaluation $\hbar=1/k$. For instance, the fiberwise product and the associated bracket in the above formulas are all given by evaluating at the complex value $\hbar=1/k$.
\end{rmk}
Since all terms in $D_\alpha$ (and also its evaluations $D_{\alpha,k}$) increase the polynomial degrees in $\Sym^\bullet\overline{T^*X}$ by $1$, we obtain the following sub-complex:
$$
\left(\Sym^\bullet\overline{T^*X}\otimes\W_X, D_{\alpha,k}\right).
$$
A nice property of this sub-complex is that it is closed under the fiberwise star product $\star_k$. 
We are now ready to define \emph{non-formal} quantizable functions:
\begin{defn}\label{definition: quantizable functions}
	A flat section of $\Sym^\bullet\overline{T^*X}\otimes \W_X$ under the Fedosov connection $D_{\alpha,k}$ is called a {\em (non-formal) quantizable function of level $k$}. It is clear that quantizable functions of level $k$ form a sheaf on $X$, and we let $\mathcal{C}_{\alpha,k}^\infty$ and $C^\infty_{\alpha,k}(X)$ denote this sheaf and the space of global quantizable functions (i.e., global sections of $\mathcal{C}_{\alpha,k}^\infty$) respectively.  
\end{defn}

Now we describe some simple properties of the non-formal Fedosov connections $D_{\alpha,k}$ and the corresponding quantizable functions of level $k$.
First of all, we consider the weight defined by the polynomial degrees of anti-holomorphic terms in $\Sym^\bullet\overline{T^*X}$ and the associated increasing filtration on the bundle $\Sym^\bullet\overline{T^*X}\otimes\W_X$. From the explicit formula of the connection $D_{\alpha,k}$, it is easy to see that it does not increase the polynomial degrees in $\Sym^\bullet\overline{T^*X}$, and thus preserves this increasing filtration. In particular, there is an associated increasing filtration on the non-formal quantizable functions of level $k$. We denote by $(\mathcal{C}_{\alpha,k}^\infty)_{N}$ the subsheaf of quantizable functions whose anti-holomorphic terms in $\W_{X,\C}$ have polynomial degrees at most $N$. 

Similar to formal Fedosov quantization, we obtain a non-formal star product:
\begin{prop}
	For every $k\in\C\setminus\{0\}$, the sheaf $\mathcal{C}_{\alpha,k}^\infty$ is closed under the star product $\star_k$ defined via the Fedosov connection $D_{\alpha,k}$. This star product is compatible with the above filtration in the sense that $(\mathcal{C}_{\alpha,k}^\infty)_{N_1}\star_k (\mathcal{C}_{\alpha,k}^\infty)_{N_2}\subset (\mathcal{C}_{\alpha,k}^\infty)_{N_1+N_2}$.
\end{prop}
\begin{proof}
	Since the fiberwise Wick product $\star_k$ is compatible with the connection $D_{\alpha,k}$, the product of two flat sections is still flat. This defines the star product $\star_k$. The statement about the filtration is obvious. 
\end{proof}
We now give some examples of quantizable functions.  

\begin{eg}
	On the flat space $\C^n$ equipped with the standard K\"ahler form, every polynomial in $\C[z^1,\bar{z}^1,\cdots,z^n,\bar{z}^n]$ is a quantizable function. 
\end{eg}

\begin{eg}\label{example: quantizable-function-evaluation}
	Given an admissible Fedosov connection $D_\alpha$, we can construct a class of non-formal quantizable functions of any level $k$ (i.e., flat sections under the non-formal Fedosov connection $D_{\alpha,k}$) by evaluating formal quantizable functions at $\hbar=1/k$. First of all, there exists the following morphism of bundles:
	$$
	 \Sym^\bullet\overline{T^*X}\otimes\W_X[\hbar]\rightarrow \Sym^\bullet\overline{T^*X}\otimes\W_X
	$$
	by taking the evaluation $\hbar=1/k$.  It is easy to see that this map preserves the Wick products and Fedosov connections on both sides. In particular, by taking cohomology with respect to $D_\alpha$ and $D_{\alpha, k}$ respectively, we obtain the following morphism of sheaves of algebras:
	\begin{equation}\label{equation: evaluation-map}
		ev_k: \mathcal{C}^\infty_{q,\hbar}\rightarrow \mathcal{C}^\infty_{\alpha,k},
	\end{equation}
    where $\mathcal{C}_{q,\hbar}^\infty$ denotes the sheaf of formal quantizable functions (Definition \ref{definition: formal-quantizable-functions}).
	 In Example \ref{example: holomorphic-functions}, we have shown that for a (local) holomorphic function $f$, the section $O_f$ is a formal quantizable function. Since it does not contain the formal variable $\hbar$, the evaluation map $ev_k$ does nothing to $O_f$ which makes it a non-formal quantizable function for any level $k$. In particular, the sheaf of quantizable functions gives a quantum $\mathcal{O}_X$-module via left multiplication. 
\end{eg}

In Section \ref{section: quantum-moment-maps}, we will see another class of quantizable functions arising from symmetries of the K\"ahler manifold $X$.

There is a type decomposition by the $1$-forms in $D_{\alpha,k}$, in which the $(1,0)$-component is independent of $k$ and explicitly given by 
$$
D_{\alpha,k}^{1,0}=\nabla^{1,0}-\delta^{1,0}. 
$$
This implies the following
\begin{lem}\label{lemm: flat-section-1-0-components}
Let $\gamma$ be a non-formal quantizable function of level $k$ (i.e., flat section with respect to the Fedosov connection $D_{\alpha,k}$). Then $\gamma$ is determined by its components in $\overline{\W}_X$. 
\end{lem}
\begin{proof}
	Let $\gamma_{0,*}$ denote the components of $\gamma$ in $\overline{\W}_X$. By Lemma \ref{lemma: flat-section-1-0-comonent}, the section 
	$$
	\tilde{\gamma}:=\sum_{k\geq 0}(\tilde{\nabla}^{1,0})^k(\gamma_{0,*})
	$$
	must be annihilated by $D_{\alpha,k}^{1,0}$. Thus 
	$$
	D_{\alpha,k}^{1,0}\left(\gamma-\tilde{\gamma}\right)=0.
	$$
	According to the construction,  $(\gamma-\tilde{\gamma})_{0,*}=0$. Suppose $\gamma-\tilde{\gamma}\not=0$. Then the terms in $D_{\alpha,k}^{1,0}\left(\gamma-\tilde{\gamma}\right)$ of the lowest polynomial degree in $\W_X$ are equal to $\delta^{1,0}\left(\gamma-\tilde{\gamma}\right)$ and cannot vanish, which is a contradiction. 
\end{proof}

In formal Fedosov quantization, there is a stronger statement, namely, a flat section is uniquely determined by its symbol. This gives the one-to-one correspondence between flat sections and formal smooth functions in Theorem \ref{theorem: original Fedosov}. It is natural to ask if this still holds for non-formal Fedosov connections and quantizable functions. However, the following example gives a negative answer to this question.
\begin{eg}
We consider the Fedosov connection $D_\alpha$ where $\alpha=\hbar\cdot\omega$. The term $J_\alpha$ in this connection is explicitly given by
\begin{align*}
J_\alpha=-\hbar\cdot\sum_{l\geq 0}(\tilde{\nabla}^{1,0})^k\left(\omega_{i\bar{j}}d\bar{z}^j\otimes y^i\right)
=-\hbar\cdot\omega_{i\bar{j}}d\bar{z}^j\otimes y^i.
\end{align*}
If we take the evaluation $\hbar=1$, then the non-formal Fedosov connection can be written explicitly as 
$$
D_{\alpha,k=1}=\nabla-\delta^{1,0}+ [I,-]_{\star_k}.
$$ 
We now construct a non-trivial flat section whose symbol actually vanishes. We claim that the local section $\omega_{i\bar{m}}\bar{y}^m$ is flat under $D_{\alpha,k=1}$. Firstly, 
$$
D_{\alpha,k=1}^{1,0}(\omega_{i\bar{m}}\bar{y}^m)=\nabla^{1,0}(\omega_{i\bar{m}}\bar{y}^m)-\delta^{1,0}(\omega_{i\bar{m}}\bar{y}^m)=0.
$$
On the other hand, the fact that $\nabla^{0,1}(\omega_{i\bar{m}}\bar{y}^m)=0$ implies the vanishing 
$$
D_{\alpha,k=1}^{0,1}(\omega_{i\bar{m}}\bar{y}^m)=0. 
$$
Thus $\omega_{i\bar{m}}\bar{y}^m$ is a local quantizable function whose symbol vanishes. 
\end{eg}

It seems then that the name ``quantizable functions'' might not be so appropriate, since the symbol of a non-trivial flat section might vanish. This name can be justified in two ways. First of all, for a large enough $k$, the above uniqueness statement remains true:
\begin{prop}\label{proposition: uniqueness-flat-section-k-big-enough}
For $|k|>>0$, a flat section $\gamma$ under the Fedosov connection $D_{\alpha,k}$ is uniquely determined by its symbol $\sigma(\gamma)$. 
\end{prop}
\begin{proof}
	Suppose $\gamma$ and $\tilde{\gamma}$ are two flat sections under $D_{\alpha,k}$ with the same symbol, such that $\xi=\gamma-\tilde{\gamma}$ is non-trivial.  By Lemma \ref{lemm: flat-section-1-0-components}, $\xi$ is determined by $\xi_{0,*}$. Let $\xi_{0,l}$ be the term in $\xi_{0,*}$ of the highest polynomial degree in $\Sym^\bullet\overline{T^*X}$. Then we have $l>0$, since otherwise, $\sigma(\xi)$ will involve a term of polynomial degree $(0,0)$, i.e., a holomorphic function, which contradicts the fact that $\gamma$ and $\tilde{\gamma}$ have the same symbols. 
	
	Let $\varphi$ be a potential of $\alpha$, i.e., $\partial\bar{\partial}\varphi=\alpha$. Recall that the term $J_\alpha$ in the Fedosov connection $D_\alpha$ is given by $-\sum_{m\geq 1}(\tilde{\nabla}^{1,0})^m(\bar{\partial}\varphi)$. An explicit computation gives that the term in $D_{\alpha,k}^{0,1}(\xi)$ of type $\A_X^{0,1}(\Sym^{l-1}\overline{T^*X})$ can be written as
	$$
	-\delta^{0,1}(\xi_{0,l})-[\tilde{\nabla}^{1,0}\left(\bar{\partial}\varphi\right)_{\hbar=1/k},\xi_{0,1}]_{\star_k}=\left(-\delta^{0,1}+O\left(\frac{1}{k}\right)\right)(\xi_{0,1}).
	$$
	Since the leading term is non-zero, if $|k|$ is sufficiently large, we obtain the non-vanishing of $D_{\alpha,k}^{0,1}(\xi)$, which contradicts the flatness of $\xi$.
\end{proof}

Therefore, for $|k|>>0$, non-formal quantizable functions of level $k$ are in a one-to-one correspondence with a sub-class of smooth functions in $C^\infty(X)$, which, by abuse of notation, will also be called \emph{non-formal quantizable functions of level $k$}. The Fedosov connection $D_{\alpha,k}$ defines a non-formal deformation of this sub-class of smooth functions. Moreover, these functions are, by construction, ``close enough'' to holomorphic ones since their associated flat sections have finite polynomial degrees in $\overline{T^*X}$. We will see in the next section that these functions form a sheaf of twisted holomorphic differential operators. 

The second justification is that, as we will show in Corollary \ref{corollary: evaluation-surjective}, the evaluation map \eqref{equation: evaluation-map} is sheaf-theoretically surjective. In other words, for any level $k$, every quantizable function of level $k$ can be locally obtained from a formal quantizable function.

\section{Sheaves of twisted differential operators from quantizable functions}\label{section: TDO}
We first give a brief review of the notion of twisted differential operators, following the notes by Ginzburg \cite{Ginzburg}.

\begin{defn}
	A filtered ring $A$ is called {\em almost commutative} if $gr A$ is commutative.
\end{defn}

\begin{defn}
A \emph{sheaf of twisted differential operators} (TDO for short) on $X$ is a positively filtered sheaf $\mathcal{D}$ of almost commutative algebras together with an isomorphism
$$
\psi_{\mathcal{D}}:gr\mathcal{D}\rightarrow \Sym^*_{\mathcal{O}_X} \mathcal{T}X
$$
of Poisson algebras, where $\mathcal{T}X$ denotes the holomorphic tangent sheaf on $X$. 
\end{defn}

\begin{rmk}
	The Poisson bracket on $\Sym^*_{\mathcal{O}_X}\mathcal{T}X$ is the natural extension of the Lie bracket on $\mathcal{T}X$. In particular, we have the maps
	$$
	\{-,-\}: \Sym^m_{\mathcal{O}_X}\mathcal{T}X\times \Sym^n_{\mathcal{O}_X}\mathcal{T}X\rightarrow \Sym^{m+n-1}_{\mathcal{O}_X}\mathcal{T}X.
	$$
\end{rmk}

The main result of this section is that the sheaf $\mathcal{C}_{\alpha,k}^\infty$ of non-formal quantizable functions of level $k$ defines a TDO. 

\subsection{A filtration on quantizable functions}
\


Consider the natural increasing filtration $(\mathcal{C}_{\alpha,k}^\infty)_0\subset(\mathcal{C}_{\alpha,k}^\infty)_1\subset\cdots$ on $\mathcal{C}_{\alpha,k}^\infty$ by polynomial degrees of anti-holomorphic terms in $\W_{X,\C}$. Let $gr\ \mathcal{C}^\infty_{\alpha,k}$ denote the associated graded sheaf.  Suppose $\alpha$ is a local flat section under $D_{\alpha,k}$ which lives in $(gr\ \mathcal{C}^\infty_{\alpha,k})_N(U)$. Let $\alpha_N\in(\W_{X,\C})_N$ denote the leading term with respect to this filtration. The proof of Lemma \ref{lemma: flat-section-1-0-comonent} says that it must be of the form
$$
\alpha_N=\sum_{k\geq 0}(\tilde{\nabla}^{1,0})^k(\alpha_{0,N}),
$$
and the following vanishing holds:
$$
\nabla^{0,1}(\alpha_{0,N})=0.
$$
\begin{prop}\label{proposition: associated-graded-of-quantizable-functions}
	For any $N\geq 0$, there is the following isomorphism of $\mathcal{O}_X$-modules: 
	\begin{align*}
		\psi:(gr\ \mathcal{C}^\infty_{\alpha,k})_N\rightarrow& \Sym^N_{\mathcal{O}_X} \mathcal{T}X\\
		\alpha\mapsto& \alpha_{0,N}\lrcorner(\omega^{-1})^N
	\end{align*}
\end{prop}
\begin{proof}
	We have seen that $\nabla^{0,1}(\alpha_N)=0$. Since $\omega$ is parallel with respect to $\nabla$, it follows that the image $\alpha_{0,N}\lrcorner(\omega^{-1})^N$ is $\bar{\partial}$-closed. Since the flat section corresponding to a holomorphic function does not contain any terms in $\overline{\W}_X$, the map $\psi$ is a morphism of $\mathcal{O}_X$-modules. 
	
	Next we show that $\psi$ is an isomorphism. The injectivity of $\psi$ is obvious. To show its surjectivity, we consider the formal functions $u_j$'s in Proposition \ref{proposition: function-u-k}. From the explicit formula in the proposition, we have
	$ev_k(O_{u_j})\in (gr\ \mathcal{C}^\infty_{\alpha,k})_1$.
	Moreover, the leading term of $ev_k(O_{u_j})$ is $\omega_{j\bar{l}}\bar{y}^l$, which implies that $\psi(ev_k(O_{u_j}))=\partial_{y^j}$.
	This finishes the proof of this proposition since $\Sym^N_{\mathcal{O}_X}\mathcal{T}X$ is locally generated by $\partial_{y^j}$'s as an $\mathcal{O}_X$-module. 
\end{proof}
\begin{cor}\label{corollary: evaluation-surjective}
The evaluation map \eqref{equation: evaluation-map} in Example \ref{example: quantizable-function-evaluation} is surjective as a morphism of sheaf of algebras. 
\end{cor}
\begin{proof}
We only need to show that locally every quantizable function can be obtained by taking evaluation of formal quantizable functions. We use an induction via the filtration on $\mathcal{C}_{\alpha,k}^\infty$: for $N=0$, we have the isomorphism $(gr\ \mathcal{C}^\infty_{\alpha,k})_0\cong\mathcal{O}_X$, and Example \ref{example: holomorphic-functions} implies that $(gr\ \mathcal{C}^\infty_{\alpha,k})_0$ lives in the image of $\psi$. Suppose $\gamma\in(gr\ \mathcal{C}^\infty_{\alpha,k})_N$ for some $N>0$. By Proposition \ref{proposition: associated-graded-of-quantizable-functions}, there exists a holomorphic function $f$ and indices $i_1,\cdots, i_N$, such that $O_f\star_k O_{u_1}\star\cdots\star O_{u_N}-\gamma\in (gr\ \mathcal{C}^\infty_{\alpha,k})_{N-1}$. Applying the induction hypothesis finishes the proof. 
\end{proof}

\begin{thm}\label{theorem: TDO}
For any $\alpha$ and level $k$, the sheaf $\mathcal{C}_{\alpha,k}^\infty$ of quantizable functions of level $k$ (i.e., flat sections under the Fedosov connection $D_{\alpha,k}$) forms a sheaf of twisted differential operators (TDO) on $X$. 
\end{thm}
\begin{proof}

 We first show that the associated product on $gr\mathcal{D}$ is commutative. Let $O_f\in \mathcal{D}_k\setminus\mathcal{D}_{k-1}$ and $O_g\in \mathcal{D}_l\setminus\mathcal{D}_{l-1}$; equivalently, the highest polynomial degrees of anti-holomorphic terms in $O_f$ and $O_g$ are $k$ and $l$ respectively. It is clear that  $O_f\star O_g, O_g\star O_f\in  \mathcal{D}_{l+k}\setminus\mathcal{D}_{l+k-1}$, and they have the same monomials of the highest anti-holomorphic polynomial degree. So we have $[O_f, O_g]_\star\in\mathcal{D}_{l+k-1}$, and hence the associated graded  $gr D$ is commutative. 
 
 To show that $\psi$ respects the Poisson structures on both sides, we only need to compute the bracket 
 $$
 [-,-]:\mathcal{D}^1\times\mathcal{D}^0\rightarrow\mathcal{D}^0.
 $$
 We take any holomorphic function $f$ on $U$ and its associated flat section $O_f$, and consider $ev_k(O_{u_j})\in\mathcal{D}^1(U)$. It is then clear that $[O_{u_j},O_f]_\star\in\mathcal{D}^0$, which must correspond to a holomorphic function given by its symbol. A simple computation shows that it is precisely given by 
 \begin{align*}
 	[\omega_{j\bar{l}}\bar{y}^l,\frac{\partial f}{\partial z^m}y^m]_\star=&\frac{\partial f}{\partial z^j}.
 \end{align*} 
This identifies the Poisson brackets on $gr\mathcal{D}$ and $\Sym^\bullet_{\mathcal{O}_X}(\mathcal{T}X)$. 
\end{proof}

\subsection{Characteristic classes of the TDO given by quantizable functions}
\

For every TDO on a complex algebraic variety, there is a characteristic class which lives in $H^1(X,\Omega^{\geq 1})$.  When $X$ is a complex manifold, this cohomology group is isomorphic to $H^1(X,\Omega_{cl}^1)$, where $\Omega_{cl}$ denotes the subsheaf of $\Omega_X^1$ which is closed under $\partial$. It is explained in \cite{Ginzburg} that locally trivial TDOs on a smooth variety are classified by this cohomology class.

We briefly recall the characteristic class of a TDO $\mathcal{D}$ on a complex manifold $X$, and refer to \cite{Ginzburg} for more details. As a TDO, there is an increasing filtration $\mathcal{D}_0\subset\mathcal{D}_1\subset\cdots \subset\mathcal{D}$ on $\mathcal{D}$, in which $\mathcal{D}_1$ is locally isomorphic to $\mathcal{O}_X\oplus \mathcal{T}_X$. Let $\{U_i\}$ be an open covering of $X$ such that over each double intersection $U_i\cap U_j$, the above isomorphisms gives rise to a map
$$
\mathcal{O}_X(U_i\cap U_j)\oplus\mathcal{T}_X(U_i\cap U_j)\rightarrow \mathcal{O}_X(U_i\cap U_j)\oplus\mathcal{T}_X(U_i\cap U_j), 
$$
which must be of the form 
\begin{equation}\label{equation: characteristic-class-TDO}
	(f,\xi)\mapsto (f+\alpha_{ij}(\xi),\xi),
\end{equation}
for some $\alpha_{ij}\in\Omega^1(U_i\cap U_j)$; here $f\in\mathcal{O}_X(U_i\cap U_j)$ and $\xi\in \mathcal{T}_X(U_i\cap U_j)$. 
In particular, on a K\"ahler manifold, if we choose the $(1,0)$-forms $\alpha_{ij}$ to be $\partial$-closed. It is easy to see that $\{\Omega_{ij}\}$ must satisfy the cocycle condition, and the \v{C}ech cohomology class $\{\alpha_{ij}\}\in H^1(X,\Omega^1)$ is defined as the \emph{characteristic class} of the TDO $\mathcal{D}$. 
  
For the sheaf $\mathcal{C}_{\alpha,k}^\infty$ of quantizable functions of level $k$, we first choose an open covering $\{U_i\}$ of $X$ by balls, and fix a potential $\rho_i$ on each $U_i$. Then on each $U_i$, the isomorphism $\mathcal{O}_X(U_i)\oplus\mathcal{T}_X(U_i)\cong \mathcal{D}_1(U_i)$ is given explicitly by 
$$
f\mapsto f, \hspace{5mm}\partial_{z^j}\mapsto O_k\left(\frac{\partial\rho_i}{\partial z^j}\right)
$$
for any holomorphic function $f$ on $U_i$. 
On each intersection $U_i\cap U_j$, we choose a $(1,0)$-form
$$
\alpha_{ij}:=\partial(\rho_i-\rho_j),
$$ 
which clearly satisfies equation \eqref{equation: characteristic-class-TDO}. 
We have $\bar{\partial}(\alpha_{ij})=0$ since $\rho_i$'s are all potentials of the same closed $(1,1)$-form, and $\partial(\alpha_{ij})=0$ since $\partial^2=0$. Thus we the \v{C}ech cohomology representative of the characteristic class of $\mathcal{C}_{\alpha,k}^\infty$ is precisely given by
$$
\{\alpha_{ij}\}\in H^1(X,\Omega_{cl}^1).
$$
On the other hand, a standard argument in complex geometry shows that the Dolbeault representative of $\{\alpha_{ij}\}$ is exactly given by $\bar{\partial}\partial \rho_i$ which gives a global closed differential form of type $(1,1)$. This exactly coincides with the Karabagov form. 

To summarize, we obtain the
\begin{thm}\label{theorem: characteristic class of TDO}
	The characteristic class of the TDO $\mathcal{C}_{\alpha,k}^\infty$ is given by $[\omega - \alpha]$.
\end{thm}


\section{Quantizing functions as holomorphic differential operators}\label{section:quantizable functions as HDOs}
In this section, we will see how quantizable functions, which come from deformation quantization, ``act on'' geometric quantization. More precisely, let $X$ be a K\"ahler manifold equipped with a prequantum line bundle $L$, meaning that there is a connection $\nabla_L$ on $L$ whose curvature is $\nabla_L^2=\omega$.  Then we will show that quantizable functions act on the holomorphic sections of $L^{\otimes k}$ as \emph{holomorphic differential operators}. This gives a non-formal quantization of a large (dense as $k \to \infty$) subspace of smooth functions on $X$.


Let $k$ be any positive integer and take $\hbar = 1/k$. We consider the sheaf $\mathcal{C}_{\alpha,k}^\infty$ of quantizable functions with level $k$, where
\begin{equation}\label{equation: choice of alpha}
	\alpha := -\hbar Ric_X = -\frac{1}{k} R_{i\bar{j}k\bar{l}}\omega^{k\bar{l}}dz^i\wedge d\bar{z}^j.
\end{equation}
Throughout this section, we fix this $\alpha$ whose associated star product is precisely the \emph{Berezin-Toeplitz quantization} with Karabegov form $k(\omega-\alpha)$. In this case, we will prove that the sheaf $\mathcal{C}_{\alpha,k}^\infty$ of quantizable functions is isomorphic to the sheaf $\mathcal{D}(L^{\otimes k})$ of holomorphic differential operators on $L^{\otimes k}$ (as TDOs). 


The main idea comes from our previous work \cite{CLL3}, where we extended Fedosov's method from algebras to modules. In this section, we will construct a sheaf of modules over the Weyl bundle $\W_{X,\C}$, which we call a {\em level $k$ Bargmann-Fock sheaf}, analogous to what we did in \cite[Section 3.2]{CLL3}. This module sheaf is equipped with a (non-formal) flat Fedosov connection $D_{\alpha, k}$, whose flat sections form a sub-sheaf. In Theorem \ref{theorem: Bargmann-Fock-isomorphic-prequantum}, we will show by local computations that this sub-sheaf is isomorphic to the sheaf of holomorphic sections of the $k$-th tensor power $L^{\otimes k}$. In particular, the space of global flat sections of the Bargmann-Fock sheaf is isomorphic to $H^0(X,L^{\otimes k})$.

We will then proceed to show that the (non-formal) Fedosov connection on the Weyl bundle is compatible with that on the Bargmann-Fock module sheaf in the sense of algebras and modules. (This also justifies our abuse of the name \emph{Fedosov connections} for various objects.)  It follows that quantizable functions, as flat sections, can act on flat sections of the Bargmann-Fock sheaf which is isomorphic to the sheaf of holomorphic sections of $L^{\otimes k}$. Since the flat sections (of both functions and sections of $L^{\otimes k}$) are local in the sense that they only depend on their infinite jets, the above action is via holomorphic differential operators. In this way, we obtain a map of sheaves from quantizable functions to holomorphic differential operators on $L^{\otimes k}$, which we will show, again by local computations, that gives an isomorphism of sheaves (Theorem \ref{theorem: almost-holomorphic-function-differential-operators}).  
Last but not the least, we will explain a generalization to holomorphic differential operators on \emph{any} line bundle as quantization of functions.

We begin with the linear algebra of the Bargmann-Fock action.
\begin{defn}\label{definition: Fock-representation-of-Wick-algebra}
We define an \emph{action} of a monomial $f = z^{\alpha_1}\cdots z^{\alpha_k}\bar{z}^{\beta_1}\cdots\bar{z}^{\beta_l} \in \mathcal{W}_{\C^n}$ on $s\in\mathcal{F}_{\C^n} := \C[[z^1,\cdots,z^n]][[\hbar]]$ by 
\begin{equation}\label{equation: Toeplitz-operators-C-n}
f\circledast s:=\left(-\hbar\right)^l\frac{\partial}{\partial z^{\beta_1}}\circ\cdots\circ\frac{\partial}{\partial z^{\beta_l}}\circ m_{z^{\alpha_1}\cdots z^{\alpha_k}}(s),
\end{equation}
where $m_{z^{\alpha_1}\cdots z^{\alpha_k}}$ denotes the multiplication by $z^{\alpha_1}\cdots z^{\alpha_k}$. 
We use the notation $\circledast_k$ to stand for the action in equation \eqref{definition: Fock-representation-of-Wick-algebra} when we take the evaluation $\hbar=1/k$. 
It is known that
$$
f\circledast(g\circledast s)=(f\star g)\circledast s.
$$
Thus equation \eqref{equation: Toeplitz-operators-C-n} defines an action of the Wick algebra $\mathcal{W}_{\C^n}$ on $\mathcal{F}_{\C^n}$, known as the {\em Bargmann-Fock representation} (or the {\em Wick normal ordering} in physics literature). 
\end{defn}

The K\"ahler form on $X$ enables us to define the fiberwise Bargmann-Fock action, making the holomorphic Weyl bundle $\W_X$ a sheaf of $\W_{X,\C}$-modules. Explicitly, a monomial in $\W_{X,\C}$ acts as a differential operator on $\W_X$ as
\begin{equation}\label{equation: Wick-ordering-formula}
 y^{i_1}\cdots y^{i_k}\bar{y}^{j_1}\cdots\bar{y}^{j_l}\mapsto \left(-\hbar\right)^l\omega^{p_1\bar{j}_1}\cdots\omega^{p_l\bar{j}_l}\frac{\partial}{\partial y^{p_1}}\circ\cdots  \frac{\partial}{\partial y^{p_l}}\circ m_{y^{i_1}\cdots y^{i_k}}.
\end{equation}

\begin{defn}
 For every positive integer $k$,  we define the \emph{level $k$ Bargmann-Fock sheaf} $\mathcal{F}_{L^{\otimes k}}$ by twisting $\W_X$ with the $k$-th tensor power of the prequantum line bundle $L$:
 $$
 \mathcal{F}_{L^{\otimes k}}:=\mathcal{W}_{X}\otimes_{\mathcal{O}_X}L^{\otimes k}.
 $$
\end{defn}

The following proposition shows that the Fedosov connection can be extended to a flat connection on $\mathcal{F}_{L^{\otimes k}}$ in a compatible way.
\begin{prop}\label{lemma: compatibility-Fedosov-connections}
Let $\nabla_{L^{\otimes k}}$ denote the Chern connection on $L^{\otimes k}$ whose curvature is $k\cdot\omega$. Then the connection
$$
(\nabla+k\cdot \gamma_{\alpha,k}\circledast_k)\otimes 1+1\otimes \nabla_{L^{\otimes k}}
$$
on the level $k$ Bargmann-Fock sheaf $\mathcal{F}_{L^{\otimes k}}$ is flat and compatible with the Fedosov connection $D_{\alpha,k}$ on $\mathcal{W}_{X,\C}$. Thus, by abuse of notation, we also denote the above connection on $\mathcal{F}_{L^{\otimes k}}$ by $D_{\alpha,k}$.
\end{prop}
\begin{proof}
Let $s$ be a section of $\W_X$. Then we have
\begin{align*}
 R_\nabla\circledast_k s & = R_{i\bar{j}p\bar{q}}dz^i\wedge d\bar{z}^j\otimes y^p\bar{y}^q\circledast_k s\\
 & = R_{i\bar{j}p\bar{q}}dz^i\wedge d\bar{z}^j\otimes \bar{y}^q\circledast_k (y^p\cdot s)\\
 & = R_{i\bar{j}p\bar{q}}dz^i\wedge d\bar{z}^j\otimes \left(-\frac{1}{k}\right)\omega^{l\bar{q}}\frac{\partial}{\partial y^l}(y^p\cdot s)\\
 & = \frac{1}{k}\left(\nabla^2(s)-R_{i\bar{j}p\bar{q}}\cdot\omega^{p\bar{q}}dz^i\wedge d\bar{z}^j\right)\circledast_k s\\
 & = \left(\frac{1}{k}\nabla^2(s)+\alpha\right)\circledast_k s.
\end{align*}
Note that we have used the definition of the Bargmann-Fock action. 
Flatness of the connection $D_{\alpha,k}$ follows from a straightforward computation:
 \begin{align*}
  D_{\alpha,k}^2=&\nabla^2\otimes 1+\left(( k\nabla\gamma_{\alpha,k}+k^2\cdot\gamma_{\alpha,k}\star_k\gamma_{\alpha,k})\circledast_k\right)\otimes 1 +1\otimes\nabla_{L^{\otimes k}}^2\\
  =&\left((k\nabla\gamma_{\alpha,k}+k^2\cdot\gamma_{\alpha,k}\star_k\gamma_{\alpha,k}+k\cdot R_\nabla-k\alpha)\circledast_k\right) \otimes 1 + 1\otimes\nabla_{L^{\otimes k}}^2\\
  =&-k\cdot\omega+k\cdot\omega = 0.
 \end{align*}
Here we have used equation \eqref{equation: Fedosov-equation-gamma} and the prequantum condition that $\nabla_{L^{\otimes k}}^2=k\cdot\omega$. To see the compatibility between Fedosov connections, let $\xi$ and $s$ be sections of $\W_{X,\C}$ and $\mathcal{F}_{L^{\otimes k}}$ respectively. Then we have
\begin{align*}
D_{\alpha,k}(\xi\circledast_k s)=&\nabla(\xi)\circledast_k s+(-1)^{|\xi|}\xi\circledast_k \nabla(s)+k\gamma_{\alpha,k}\circledast_k(\xi\circledast_k s)\\
=&\nabla(\xi)\circledast_k s+(-1)^{|\xi|}\xi\circledast_k \nabla(s)+k[\gamma_{\alpha,k},\xi]_{\star_k}\circledast_k s)+(-1)^{|\xi|}\xi\circledast_k (k\gamma_{\alpha,k}\circledast_k s)\\
=&D_{\alpha,k}(\xi)\circledast_k s+(-1)^{|\alpha|}\xi\circledast_k D_{\alpha,k}(s).
\end{align*}
\end{proof}

In Fedosov quantization \cite{Fed}, we have a ono-to-one correspondence between smooth functions on $X$ and flat sections of the Weyl bundle. The following theorem gives the analogous correspondence for the Bargmann-Fock sheaves.
\begin{thm}\label{theorem: Bargmann-Fock-isomorphic-prequantum}
Suppose that $X$ is a K\"ahler manifold equipped with a prequantum line bundle $L$, and we choose the formal closed $(1,1)$-form $\alpha$ as in \eqref{equation: choice of alpha}. Then for any positive integer $k$, the symbol map gives a sheaf isomorphism from the sheaf of flat sections of the level $k$ Bargmann-Fock sheaf $\mathcal{F}_{L^{\otimes k}}$ with respect to the flat connection $D_{\alpha,k}$ to the sheaf of holomorphic sections of $L^{\otimes k}$.
\end{thm}
\begin{proof}
	We first fix a sufficiently fine open cover of $X$ so that local holomorphic frames of $L^{\otimes k}$ exist on any open set in the cover. Let $U \subset X$ be any open set in such a cover.

	Recall that the symbol map
	\begin{equation*}
		\sigma: \Gamma(U,\mathcal{F}_{L^{\otimes k}})\rightarrow \Gamma(U,L^{\otimes k}). 
	\end{equation*}
	is defined simply by setting all $y^i$'s to zero. 
	We write a section $s\in\Gamma(U,\mathcal{F}_{L^{\otimes k}})$ in local coordinates as $s=\sum_J s_Jy^J\otimes e_{L^k}$, where $J$ runs over all holomorphic multi-indices and $e_{L^k}$ is a local holomorphic frame of $L^{\otimes k}$. Then
\begin{align*}
 D_{\alpha,k}(s)&=D_{\alpha,k}(\sum_{|J|\geq 0}s_Jy^J\otimes e_{L^k})\\
       &=d_X(s_0)\otimes e_{L^k}+\sum_{|J|>0}k\cdot\gamma_{\alpha,k}\circledast_k(s_Jy^J\otimes e_{L^k})+(\sum_{|J|\geq 0}s_Jy^J)\otimes \nabla_L(e_{L^k})\\
       &=\left(\partial_X(s_0)+\bar{\partial}_X(s_0)\right)\otimes e_{L^k}+\sum_{|J|>0}k\cdot\gamma_{\alpha,k}\circledast_k(s_Jy^J\otimes e_{L^k})+(\sum_{|J|\geq 0}s_Jy^J)\otimes \nabla_L^{1,0}(e_{L^k}),
\end{align*}
where $d_X$ denotes the de Rham differential on $X$.

By analyzing the type of the part $\gamma_{\alpha,k}$ in the Fedosov connection, it is easy to see that if $\bar{\partial}s_0\not=0$, then we must have $D_{\alpha,k}(s)\neq 0$. Thus, the symbols of flat sections $\Gamma^{flat}(U,\mathcal{F}_{L^{\otimes k}})$ must lie in $H^0(U,L^{\otimes k})$. This induces the map
\begin{equation}\label{equation: symbol-map-Bargmann-Fock-sheaf}
\sigma: \Gamma^{flat}(U,\mathcal{F}_{L^{\otimes k}})\rightarrow H^0(U,L^{\otimes k}). 
\end{equation}

To show the surjectivity of this map, we first find a flat section of the Bargmann-Fock sheaf $\mathcal{F}_{L^{\otimes k}}$ whose symbol is $e_{L^k}$. Suppose that the hermitian metric of $L^{\otimes k}$ is given locally by $\langle e_{L^{k}},e_{L^{k}}\rangle=e^{k\cdot\rho}$. The connection $\nabla_{L^{\otimes k}}$ can then be written explicitly as
$$
\nabla_{L^{\otimes k}}(e_{L^{k}})=k\cdot\partial\rho\otimes e_{L^{k}},
$$
and the prequantum condition implies that $\bar{\partial}\partial\rho=\omega$. 
We define a local section of the holomorphic Weyl bundle $\W_X$ by $\beta:=\sum_{k\geq 1}(\tilde{\nabla}^{1,0})^k(\rho)$. It is clear the section $e^{k\cdot\beta}\otimes e_{L^{k}}$ of $\mathcal{F}_{L^{\otimes k}}$ has symbol $e_{L^k}$. The following proposition, which will be proved in Appendix \ref{appendix: flat-section-Bargmann-Fock}, says that this section is indeed flat under $D_{k,\alpha}$:
\begin{prop}\label{lemma: flat-section-Bargmann-Fock}
The section $e^{k\cdot\beta}\otimes e_{L^{k}}$ is closed under the Fedosov connection, i.e.,
$$
D_{\alpha,k}(e^{k\cdot\beta}\otimes e_{L^{k}})=0. 
$$
\end{prop}
The surjectivity of the symbol map \eqref{equation: symbol-map-Bargmann-Fock-sheaf} now follows from this proposition since any local holomorphic section of $L^{\otimes k}$ is of the form $f\cdot e_{L^{k}}$ for some holomorphic function $f$.  We simply take the section of $\mathcal{F}_{L^{\otimes k}}$ to be $O_f\cdot \left(e^{k\cdot\beta}\otimes e_{L^{k}}\right)$, which is obviously the image of $f\cdot e_{L^k}$. 

To prove injectivity, suppose there is a nonzero section $s\in\Gamma(U,\mathcal{F}_{L^{\otimes k}})$ such that $\sigma(s)=0$. Writing $s=\sum_{|J|=i>0}^\infty s_Jy^J$, then the lowest polynomial degree term of $D_{\alpha,k}(s)$ is given by $\delta(\sum_{|J|=i}s_Jy^J)\not=0$, which implies the non-vanishing of $D_{\alpha,k}(s)$. 
\end{proof}

\begin{rmk}
Theorem \ref{theorem: Bargmann-Fock-isomorphic-prequantum} implies that the space of global flat sections of the Bargmann-Fock sheaf $\mathcal{F}_{L^{\otimes k}}$ is isomorphic to $H^0_{\bar{\partial}}(X, L^{\otimes k})$, i.e., the Hilbert space in geometric quantization. It is now clear that the name ``level'' of the positive integer $k$ in this paper is consistent with the notion of ``level'' in quantization of gauge theory \cite[equation (2.5)]{Witten}:
$$
\omega=k\cdot\omega_0,
$$
where $k$ denotes the scaling of the original symplectic form $\omega_0$, which is equivalently to taking the $k$-th tensor power of the pre-quantum line bundle $L$. 
\end{rmk}
\begin{thm}\label{theorem: almost-holomorphic-function-differential-operators}
Suppose that $X$ is a K\"ahler manifold equipped with a prequantum line bundle $L$, and we choose the formal closed $(1,1)$-form $\alpha$ as in \eqref{equation: choice of alpha}.
Then for any positive integer $k$, there is a natural isomorphism 
$$
 \varphi: \mathcal{C}_{\alpha,k}^\infty\rightarrow \mathcal{D}(L^{\otimes k})
 $$
from the sheaf of algebras of level $k$ quantizable functions to the sheaf of holomorphic differential operators on $L^{\otimes k}$. Furthermore, this isomorphism is compatible with the filtration on quantizable functions and that on differential operators by orders, and hence gives an isomorphism of TDOs.
\end{thm}
\begin{proof}
We first define the map $\varphi$ by showing that quantizable functions act on the space of holomorphic sections of $L^{\otimes k}$ as differential operators. We have seen from Theorem \ref{theorem: Bargmann-Fock-isomorphic-prequantum} that level $k$ quantizable functions correspond to flat sections of $\mathcal{W}_{X,\C}$ under $D_{\alpha,k}$, and holomorphic sections of $L^{\otimes k}$ correspond to flat sections of $\mathcal{F}_{L^{\otimes k}}$. Since the flat connections on these two bundles are compatible, the outcome of this action is also flat and thus correspond to a holomorphic section of $L^{\otimes k}$. As to the compatibility of filtrations on both sides, it is clear since $\bar{y}^j$ acts as the differential $-\omega^{i\bar{j}}\frac{\partial}{\partial y^i}$ in the Bargmann-Fock action, and thus the filtration on both sides is preserved by $\varphi$. The locality is clear from the construction. 
 
We will perform local computations to show that $\varphi$ is an isomorphism of sheaves. 
For the injectivity of $\varphi$, suppose $f$ is a quantizable function such that $\varphi(f)=0$ as a differential operator. To show that $f=0$, we use induction on the anti-holomorphic polynomial degrees of $f$. Suppose $f\in (\mathcal{C}_{\alpha,k}^\infty)_0$, or equivalently, $f$ is a holomorphic function (recall that there is a natural increasing filtration $(\mathcal{C}_{\alpha,k}^\infty)_0\subset(\mathcal{C}_{\alpha,k}^\infty)_1\subset\cdots$ on $\mathcal{C}_{\alpha,k}^\infty$ by polynomial degrees of anti-holomorphic terms in $\W_{X,\C}$). Then $\varphi(f)$ simply acts by multiplication and thus $f=0$. For the induction step, suppose $f\in (\mathcal{C}_{\alpha,k}^\infty)_m$ and $\varphi(f)=0$. We take any local flat section $s$ of $\mathcal{F}_{L^{\otimes k}}$, and any non constant holomorphic function $g$. Then we have
 \begin{align*}
  O_f\circledast_k(O_g\cdot s)=[O_f,O_g]_\star\circledast_k s+O_g\circledast_k(O_f\circledast_k s)
  =[O_f,O_g]_{\star_k}\circledast_k s
  =0,
 \end{align*}
which implies that $\varphi(f\star_k g-g\star_k f)=0$. Since $O_g$ contains only monomials in $\W_X$ and the bracket $[-,-]_{\star_k}$ kills at least one $\bar{y}^j$'s in $O_f$, we see that $[O_f,O_g]_{\star_k}\in (\mathcal{C}_{\alpha,k}^\infty)_{m-1}$, and we have $f\star_k g-g\star_k f=0$ by the induction hypothesis. This implies that $O_f$ is also a section of $\W_X$ and has to be zero. 
 
Next we show the surjectivity of $\varphi$. With respect to the choice of the holomorphic frame $e_{L^k}$, the holomorphic differential operators are generated by holomorphic functions and $\partial_{z^i}$'s. Let $e^{k\cdot\beta}\otimes e_{L^k}$ be the section of $\mathcal{F}_{L^{\otimes k}}$ as in the proof of Theorem \ref{theorem: Bargmann-Fock-isomorphic-prequantum}. We consider the function $u_i$ defined in Propossition \ref{proposition: function-u-k}, and claim that $ ev_k(O_{u_i})\circledast_k(e^{k\cdot\beta}\otimes e_{L^k})=0$. Note that here we are taking the evaluation of $O_{u_i}$ at $\hbar=1/k$ using the map $ev_k$ in equation \eqref{equation: evaluation-map}.  By Theorem \ref{theorem: Bargmann-Fock-isomorphic-prequantum}, we only need to show that its symbol vanishes. 

For this purpose, notice that, as shown in Proposition \ref{proposition: function-u-k}, every term in $ev_k(O_{u_i})$ has polynomial degree at most $1$ in the anti-holomorphic $\bar{y}^i$'s. Thus we only need terms in $ev_k(O_{u_i})$ of type $(0,0)$, $(0,1)$ and $(1,1)$ to find the symbol $\sigma\left(ev_k(O_{u_i})\circledast_k(e^{k\cdot\beta}\otimes e_{L^k})\right)$. For this, we recall that 
$$
ev_k(O_{u_i})=\left(\frac{\partial\rho}{\partial z^i}+2\sqrt{-1}\omega_{i\bar{k}}\bar{y}^k+2\sqrt{-1}\frac{\partial\omega_{i\bar{j}}}{\partial z^k}y^k\bar{y}^j\right)+\frac{1}{k}\frac{\partial\rho_1}{\partial z^i}+\cdots;
$$
here $\rho$ and  $\rho_1$ are local potentials for $\omega$ and the Ricci form respectively. The dots denote those terms which contribute trivially to the symbol. Then we compute:
\begin{align*}
 &\ \sigma\left(ev_k(O_{u_i})\circledast_k(e^{k\cdot\beta}\otimes e_{L^k})\right)\\
 =&\ \left(\frac{\partial\rho}{\partial z^i}+\frac{1}{k}\frac{\partial\rho_1}{\partial z^i}+2\sqrt{-1}\sigma\left([\omega_{i\bar{m}}\bar{y}^m,\beta]_{\star_k}\right)+2\sqrt{-1}\frac{1}{k}\frac{\partial\omega_{i\bar{j}}}{\partial z^k}\omega^{k\bar{j}}\right)\cdot(e^{k\cdot\beta}\otimes e_{L^{\otimes k}})\\
 =&\ \frac{\partial\rho}{\partial z^i}+2\sqrt{-1}[\omega_{i\bar{m}}\bar{y}^m,\frac{\partial\rho}{\partial z^j}y^j]_{\star_k}+\frac{1}{k}\left(\frac{\partial\rho_1}{\partial z^i}+\frac{\partial\omega_{i\bar{j}}}{\partial z^k}\omega^{k\bar{j}}\right)\\
 =&\ 0,
\end{align*}
where we have used the basic fact in K\"ahler geometry that 
$$\frac{\partial\rho_1}{\partial z^i}+\frac{\partial\omega_{i\bar{j}}}{\partial z^k}\omega^{k\bar{j}}=0$$
in the second equality.

On the other hand, any local flat section of $\mathcal{F}_{L^{\otimes k}}$ can be written as $O_g\cdot\left( e^{k\cdot\beta}\otimes e_{L^k}\right)$ for some holomorphic function $g$. We have
\begin{align*}
ev_k(O_{u_i})\circledast_k(O_g\cdot e^{k\cdot\beta}\otimes e_{L^k})=&\ [ev_k(O_{u_i}),O_g]_{\star_k}\circledast_k (e^{k\cdot\beta}\otimes e_{L^k})+O_g\cdot(ev_k(O_{u_i})\circledast_k e^{k\cdot\beta}\otimes e_{L^k})\\
=&\ [ev_k(O_{u_i}),O_g]_{\star_k}\circledast_k (e^{k\cdot\beta}\otimes e_{L^k})\\
=&\ -\frac{1}{k}O_{\frac{\partial g}{\partial z^i}}\cdot (e^{k\cdot\beta}\otimes e_{L^k}).
\end{align*}
In the last equality we have used Lemma \ref{lemma:bracket of flat sections} below. This shows that  $\varphi(u_i)=-\frac{1}{k}\cdot\partial_{z^i}$ under the holomorphic frame $e_{L^k}$. The proof of the theorem is now completed.
\end{proof}

\begin{lem}\label{lemma:bracket of flat sections}
For any local holomorphic function $g$ on an open set where $u_i$ also exists, there is the following equality:
$$
[ev_k(O_{u_i}), O_g]_{\star_k}=-\frac{1}{k}\cdot O_{\frac{\partial g}{\partial z^i}}.
$$
Equivalently, taking bracket of flat sections of holomorphic functions with $O_{u_i}$ is equivalent to taking the partial derivative $-\frac{1}{k}\partial_{z^i}$. 
\end{lem}
\begin{proof}
It is clear that the bracket $[ev_k(O_{u_i}), O_g]_{\star_k}$ is still a flat section. Thus we only need to find its symbol. Since $g$ is a local holomorphic function,  $O_g$ only contains terms in $\W_X$ and we only need the purely anti-holomorphic terms of $ev_k(O_{u_i})$ to compute $\sigma\left([ev_k(O_{u_i}), O_g]_{\star_k}\right)$, i.e.,
\begin{align*}
 \sigma\left([ev_k(O_{u_i}), O_g]_{\star_k}\right) & = [2\sqrt{-1}\omega_{i\bar{j}}\bar{y}^j,\frac{\partial g}{\partial z^m}y^m]_{\star_k}\\
 & = 2\sqrt{-1}\cdot\omega_{i\bar{j}}\cdot\frac{\sqrt{-1}}{2k}(-1)\omega^{m\bar{j}}\cdot\frac{\partial g}{\partial z^m}  
 = -\frac{1}{k}\cdot\frac{\partial g}{\partial z^i}.
\end{align*}
\end{proof}

Theorem \ref{theorem: almost-holomorphic-function-differential-operators} says that the sheaf $\mathcal{C}^\infty_{\alpha,k}$ consists of exactly those functions that can be quantized to differential operators on the Hilbert spaces $H^0(X,L^{\otimes k})$. This produces a non-formal deformation of the classical multiplication. The isomorphism in Theorem \ref{theorem: almost-holomorphic-function-differential-operators} also gives local generators of $\mathcal{C}^\infty_{\alpha,k}$.
\begin{cor}
For any open set $U\subset X$ isomorphic to an open ball in $\C^n$, quantizable functions on $U$ are generated by holomorphic functions and the functions $u_i$'s defined in Proposition \ref{proposition: function-u-k}. 
\end{cor}


All the above computations and results can be generalized to the following situation: Assume that the Karabegov form is admissible and \emph{integral}, namely, 
\begin{equation}\label{equation: integrality condition}
\omega_\hbar=\frac{1}{\hbar}\left(\omega - \alpha\right)\in\frac{1}{\hbar}\cdot \left(H^{1,1}(X)\cap H^2(X,\mathbb{Z})\right)[\hbar].
\end{equation}
Then we can take the evaluation $\hbar=1$, and twist $\W_{X}$ by a line bundle to obtain a Bargmann-Fock bundle which admits a Fedosov flat connection. Thus the quantizable functions associated to these Karabegov forms can be quantized to holomorphic differential operators on a line bundle. Conversely, holomorphic differential operator on any holomorphic line bundle can be identified with a class of quantizable functions. 

\begin{rmk}
	The construction here can be generalized straightforwardly to the non-abelian case: We can twist the Weyl bundle with any holomorphic vector bundle $E$ over the K\"ahler manifold $X$, and identify holomorphic differential operators $\mathcal{D}(E,E)$ with a subspace of $C^\infty(X, \End(E))$. By combining this result with the BV quantization method in \cites{GLL, CLL}, we can give another proof of the trace formula for differential operators for K\"ahler manifolds as first introduced in \cite{EF}. 
\end{rmk}

\subsection{Quantizable functions in geometric quantization}
\

In geometric quantization,  the prequantum operator $Q_f$ associated to a smooth function $f$ on sections of the $k$-th tensor power of the prequantum line bundle $L$ is defined as
$$
Q_f:=\frac{\sqrt{-1}}{2\pi\cdot k}\nabla^k_{X_f}+f,
$$
where $\nabla^k$ denotes the connection on the line bundle $L^{\otimes k}$. For a holomorphic section $s\in H^0(X,L^{\otimes k})$,  the output $Q_f(s)$ is in general \emph{not} a holomorphic section of $L^{\otimes k}$. 
\begin{defn}
	A smooth function $f \in C^\infty(X)$ is called {\em quantizable in the sense of geometric quantization} if the operator $Q_f$ preserves the Hilbert space $H^0(X,L^{\otimes k})$.  
\end{defn}
It is clear from the definition that $Q_f$ is a differential operator. Theorem \ref{theorem: almost-holomorphic-function-differential-operators} implies that there exists a quantizable function whose action can be identified with $Q_f$. So our notion of quantizable functions is a vast generalization of the previous notion of quantizable functions (or \emph{polarization-preserving functions}) in geometric quantization. In particular, we can obtain higher order differential operators from them. Furthermore, we get a subspace of smooth functions closed under the star product.

\section{Examples: first order quantizable functions from symmetries}\label{section: quantum-moment-maps}

In this section, we give a class of examples of (first order) quantizable functions arising from symmetries on K\"ahler manifolds. It is known that symmetries of symplectic manifolds is encoded in moment maps.
More precisely, let $G$ be a Lie group and $\mathfrak{g}$ be its Lie algebra. Let $(X,\omega)$ be a symplectic manifold which admits a Hamiltonian $G$-action. For every $g\in\g$, let $V_g$ denote the vector field associated to the corresponding infinitesimal action, whose action on smooth functions can be expressed as a Poisson bracket $\mathcal{L}_{V_g}=\{\mu(g),-\}$, where $\mu:\mathfrak{g}\rightarrow C^\infty(X)$ is the classical moment map. A quantized notion of the moment map, called {\em the quantum moment map}, was introduced in \cite{Xu}.

\begin{defn}\label{definition: quantum-moment-map}
 Let $X,G$ be as above. Suppose $(C^\infty(X)[[\hbar]],\star)$ is a $G$-invariant deformation quantization of $X$. Then a {\em quantum moment map} is a homomorphism of Lie algebras
	$$
	\mu_\hbar: \g\rightarrow C^\infty(X)[[\hbar]],
	$$
such that for every $g\in\g$, we have the equality $\mathcal{L}_{V_g}=[\mu_\hbar(g),-]_\star$ for formal smooth functions $C^\infty(X)[[\hbar]]$; here $V_g$ denotes the vector field associated to the infinitesimal action as above and the Lie bracket on the right hand side is the one associated to the star product $\star$. Explicitly, for $g,h\in\mathfrak{g}$, we require
	$$
	\mu_\hbar([g,h])=\mu_\hbar(g)\star\mu_\hbar(h)-\mu_\hbar(h)\star\mu_\hbar(g).
	$$
\end{defn}

We will focus on the case when $X$ is a K\"ahler manifold and the $G$-action \emph{also preserves the complex structure}. For later computations, we will fix a basis $\{g_i\}_{i=1}^{\dim\g}$ of the Lie algebra $\g$, and let $\mathcal{L}_i$ and $\iota_i$ denote respectively the Lie derivative and contraction associated to the vector field $V_{g_i}$. We will extend both the operators $\mathcal{L}_i$ and $\iota_i$ to  $\A_X^\bullet(\W_{X,\C})$. 
The operators $\iota_i$ only contract with the differential forms in $\A_X^\bullet(\W_{X,\C})$, while the operator $\mathcal{L}_i$ is extended as a derivation with respect to the super commutative product on $\A_X^\bullet(\W_{X,\C})$ (i.e., the wedge product on $\A_X$ and the commutative product on $\W_{X,\C}$).


We would like to show that the images of the quantum moment map are all first order quantizable functions. To start with, recall that the Fedosov connection $D_F$ is of the form
$$
D_F=\nabla+\frac{1}{\hbar}[I,-]_\star,
$$
where $[-,-]_\star$ is the Lie bracket associated to the fiberwise star product. The following simple computation shows that $[D_F-\nabla,\iota_i]=\frac{1}{\hbar}[\iota_i(I),-]_\star$: Assuming that a section of $\A_X(\W_{X,\C})$ is of the form $\alpha\otimes\beta$ where $\alpha\in\A^*(X), \beta\in\W_{X,\C}$, then we have
\begin{align*}
&[D_F-\nabla,\iota_i](\alpha\otimes\beta)\\
=&(D_F-\nabla)(\iota_i(\alpha)\otimes\beta)+\iota_i\circ(D_F-\nabla)\left(\alpha\otimes\beta\right)\\
=&\frac{1}{\hbar}[I,\iota_i(\alpha)\otimes\beta]_\star+\iota_i\left(\frac{1}{\hbar}[I,\alpha\otimes\beta]_\star\right)\\
=&\frac{1}{\hbar}\left(I\star(\iota_i(\alpha)\otimes\beta)-(-1)^{|\alpha|-1}(\iota_i(\alpha)\otimes\beta)\star I+\iota_i(I\star(\alpha\otimes\beta)-(-1)^{|\alpha|}(\alpha\otimes\beta)\star I)\right)\\
=&\frac{1}{\hbar}\left(\iota_i(I)\star(\alpha\otimes\beta)-(-1)^{2|\alpha|}(\alpha\otimes\beta)\star\iota_i(I)\right)\\
=&\frac{1}{\hbar}[\iota_i(I),\alpha\otimes\beta]_\star
\end{align*}

We will give an explicit expression of the operators $\mathbb{A}_i:=\mathcal{L}_i-[D_F,\iota_i]$ on the Weyl bundle. First, we have the following lemma.
\begin{lem}
	For every $1\leq i\leq\dim(G)$, the operator $\mathcal{L}_i-[D_F,\iota_i]$ is linear over $\A^\bullet(X)$; equivalently, we have
	$\mathcal{L}_i-[D_F,\iota_i]\in\Gamma(X,\End(\W_{X,\C}))$.
\end{lem}
\begin{proof}
	Let $s\in\A^\bullet(X,\W_{X,\C})$, and let $\alpha\in\A^k(X)$. Then
	\begin{align*}
		&\left(\mathcal{L}_i-[D_F,\iota_i]\right)(\alpha\wedge s)\\
		=&\mathcal{L}_i(\alpha)\wedge s+\alpha\wedge\mathcal{L}_i(s)-D_F\circ\iota_i(\alpha\wedge s)-\iota_i\circ D_F(\alpha\wedge s)\\
		=&\mathcal{L}_i(\alpha)\wedge s+\alpha\wedge\mathcal{L}_i(s)-D_F(\iota_i(\alpha)\wedge s+(-1)^k\cdot\alpha\wedge\iota_i(s))\\
		&-\iota_i( d_X(\alpha)\wedge s+(-1)^k\cdot\alpha\wedge D_F(s))\\
		=&\mathcal{L}_i(\alpha)\wedge s+\alpha\wedge\mathcal{L}_i(s)-d_X(\iota_i(\alpha))\wedge s+(-1)^k\cdot\iota_i(\alpha)\wedge D_F(s)\\
		&-(-1)^k\cdot d_X(\alpha)\wedge\iota_i(s)-\alpha\wedge D_F(\iota_i(s))\\
		&-\iota_i(d_X(\alpha))\wedge s+(-1)^k\cdot d_X(\alpha)\wedge\iota_i(s)-(-1)^k\iota_i(\alpha)\wedge D_F(s)-\alpha\wedge\iota_i(D_F(s))\\
		=&\alpha\wedge\mathcal{L}_i(s)-\alpha\wedge([D_F,\iota_i](s)).
	\end{align*}
\end{proof}
  In a similar way, we can show that the operator $\mathcal{L}_i-[\nabla,\iota_i]$ is linear over $\A^\bullet(X)$. More importantly, we will show that this operator on the Weyl bundle $\W_{X,\C}$ can be expressed as a bracket with respect to the Wick product. 
  
\begin{lem}\label{proposition: Lie-derivative-bracket}
	There exists a section $s_i$ of the Weyl bundle such that $\mathcal{L}_i-[D_F,\iota_i]$ can be written as a bracket:
	\begin{equation}\label{equation: flat-section-s-i}
	\mathcal{L}_i-[D_F,\iota_i]=\frac{1}{\hbar}[s_i,-]_\star.
	\end{equation}
\end{lem}
\begin{proof}
It is clear that $[D_F-\nabla,\iota]=[[I,-]_\star,\iota_i]=[\iota_i(I),-]_\star$. Thus we only need to show that $\mathcal{L}_i-[\nabla,\iota_i]$ can be written as a bracket with respect to the Wick product on $\W_{M,\C}$. 

For the following explicit local computations, we will use real coordinates on $X$ and $\W_{X,\C}$: Let $(x^1,\cdots, x^{2n})$ be local real coordinates on $X$, with $\eta^\alpha$'s the corresponding sections of $TX^*_\R$. 
 Let $f_i=\mu(g_i)$, where $\mu:\g\rightarrow C^\infty(X)$ is the classical moment map. Then the vector field associated to $g_i$ is $V_{g_i}=\frac{\partial f_i}{\partial x^j}\omega^{jk}\frac{\partial}{\partial x^k} $. Then we have
	\begin{align*}
		[\nabla,\iota_i](\eta^\alpha)=\iota_i(\nabla \eta^\alpha)
		= \iota_i(\Gamma_{\beta\gamma}^\alpha dx^\beta\otimes \eta^\gamma)
		=\iota_i(dx^\beta)\cdot\Gamma_{\beta\gamma}^\alpha \eta^\gamma
		= \frac{\partial f_i}{\partial x^{j}}\omega^{j\beta}\cdot\Gamma_{\beta\gamma}^\alpha \eta^\gamma,
	\end{align*}
	where $\Gamma^\alpha_{\beta\gamma}$'s are the Christoffel symbols of $\nabla$. On the other hand, using Cartan's formula, we have
	\begin{align*}
		\mathcal{L}_i(dx^\alpha)=&[d_X,\iota_{i}](dx^\alpha)
		=d_X(\iota_{i}(dx^\alpha))\\
		=&d_X\left(\frac{\partial f_i}{\partial x^j}\omega^{j\alpha}\right)
		=\left(\frac{\partial^2 f_i}{\partial x^k\partial x^j}\omega^{j\alpha}+\frac{\partial f_i}{\partial x^j}\frac{\partial\omega^{j\alpha}}{\partial x^k}\right)dx^k.
	\end{align*}
	
	Since $\mathcal{L}_i-[\nabla,\iota_i]$ is linear over $\A^\bullet(X)$, via the above computations, we can write it as
	$$
	\mathcal{L}_i-[\nabla,\iota_i]=\left(\frac{\partial^2 f_i}{\partial x^\gamma\partial x^\beta}\omega^{\beta\alpha}+\frac{\partial f_i}{\partial x^\beta}\frac{\partial\omega^{\beta\alpha}}{\partial x^\gamma}-\frac{\partial f_i}{\partial x^{j}}\omega^{j\beta}\cdot\Gamma_{\beta\gamma}^\alpha \right)\partial_{x^\alpha}\otimes \eta^\gamma\in\Gamma(X, TX_\R\otimes T^*X_\R).
	$$
	We use the symplectic form to lift the subscript in $\partial_{x^\alpha}$ in the above section, and obtain the following tensor $t_i$:
	\begin{align*}
		t_i=&\left(\frac{\partial^2 f_i}{\partial x^\gamma\partial x^\beta}\omega^{\beta\alpha}+\frac{\partial f_i}{\partial x^\beta}\frac{\partial\omega^{\beta\alpha}}{\partial x^\gamma}-\frac{\partial f_i}{\partial x^{j}}\omega^{j\beta}\cdot\Gamma_{\beta\gamma}^\alpha \right)\omega_{\alpha\xi}\eta^\xi\otimes \eta^\gamma\in\Gamma(X, T^*X_\R\otimes T^*X_\R)\\
		=&\left(\frac{\partial^2 f_i}{\partial x^\gamma\partial x^\xi}+\frac{\partial f_i}{\partial x^\beta}\frac{\partial\omega^{\beta\alpha}}{\partial x^\gamma}\omega_{\alpha\xi}-\frac{\partial f_i}{\partial x^{j}}\omega^{j\beta}\cdot\Gamma_{\beta\gamma}^\alpha\omega_{\alpha\xi} \right)\eta^\xi\otimes \eta^\gamma\in\Gamma(X, T^*X_\R\otimes T^*X_\R).
	\end{align*}
	We claim that the tensor $t_i$ is symmetric in the two indices $\xi$ and $\gamma$.  The first term clearly satisfies this symmetry. For the second and third terms, we can choose $(x^1,\cdots, x^{2n})$ to be local Darboux coordinates. Then the second term vanishes since they are derivatives of $\omega_{\beta\alpha}$. 
	
	The third term also has the desired symmetry, and we give a brief argument here: The Levi-Civita connection $\nabla$ is a symplectic connection on $X$ (namely, it is compatible with the symplectic form $\omega$ and torsion-free), and its Christoffel symbols satisfies the following symmetry:
	$$
	\Gamma_{\beta\gamma}^\alpha\omega_{\alpha\xi}
	$$
	is symmetric in all three indices $\beta, \gamma$ and $\xi$ (we refer to \cite[Definition 2.3]{Fed} for the definition and properties of symplectic connections).  
	
	This symmetry property implies that we can identify the tensor $t_i$ as a section in the Weyl bundle $\W_{X,\C}$ of polynomial degree $2$. From the definition of $t_i$, there is 
	$$
	(\mathcal{L}_i-[\nabla,\iota_i])(\eta^\alpha)=\frac{1}{\hbar}[t_i,\eta^\alpha]_\star. 
	$$
	This also implies the uniqueness of $t_i$, thus although the computations are local, they actually glue to a global section of the Weyl bundle. 
	
	For a general section $\xi$ of the Weyl bundle, we have 
	$$
\frac{1}{\hbar}[t_i,\xi]_\star=(\mathcal{L}_i-[\nabla,\iota_i])(\xi)+O(\hbar).
	$$
	 We will show that the $O(\hbar)$ term in the above expression actually vanishes. For this, we write $t_i$ in terms of complex coordinates $(z^1,\cdots, z^n)$ and corresponding sections $y^\alpha$'s. Since $\mathcal{L}_i$ preserves complex structure, we must have that $\mathcal{L}_i$ preserves types in $T^*X_\C$, and so is $[\nabla,\iota_i]$, since the Levi-Civita connection $\nabla$ is compatible with the complex structure. Thus $t_i$ must be of the form
	$$
	t_i=(t_i)_{\alpha\bar{\beta}}y^\alpha\bar{y}^\beta. 
	$$
	The type of $t_i$ implies the vanishing of the $O(\hbar)$ term, by the definition of the product $\star$ in equation \eqref{equation: fiberwise-Wick-product}. 
\end{proof}
\begin{prop}\label{corollary: Lie-derivative-bracket}
Let $\alpha\in\Gamma^{flat}(X,\W_{X,\C}[[\hbar]])\cong C^\infty(X)[[\hbar]]$ be any flat section of the Weyl bundle, which corresponds to a formal function on $X$. Then we have
	$$
	s_i\star\alpha-\alpha\star s_i=\mathcal{L}_i(\alpha).
	$$
	In other words, quantum Hamiltonian symmetries $\mathcal{L}_i$ on formal functions can be expressed as brackets with $s_i$.
\end{prop}
\begin{proof}
	From Lemma \ref{proposition: Lie-derivative-bracket}, we have
	\begin{align*}
	s_i\star\alpha-\alpha\star s_i=(\mathcal{L}_i-[D_F,\iota_i])(\alpha)
	=\mathcal{L}_i(\alpha)-\iota_i(D_F(\alpha))
	=\mathcal{L}_i(\alpha).
	\end{align*}
\end{proof}
We give the following lemma, which we will need later.
\begin{lem}
	If the $G$-action on $X$ preserves both the symplectic and complex structures (i.e. holomorphic isometries), then the Fedosov connection $D_F$ is $G$-invariant. In particular, for all $1\leq i\leq\dim(G)$, we have
	\begin{equation}\label{equation: Fedosov-connection-invariant}
	[\mathcal{L}_i,D_F]=0.
	\end{equation}
\end{lem}
\begin{proof}
	Recall that the Fedosov connection is of the explicit form $D_F=\nabla+\frac{1}{\hbar}[I,-]_\star$. The Levi-Civita connection $\nabla$ obviously commutes with the $G$-action. On the other hand, the components of the term $I$ in the Fedosov connection arises by iteratively applying the operators $(\delta^{1,0})^{-1}$ and $\nabla^{1,0}$ to the curvature operator $\nabla^2$ and the Ricci curvature. The result now follows because all of these commute with the $G$-action. 
\end{proof}

\begin{lem}
	Locally, by adding formal smooth functions (the constant terms in $\W_{X,\C}[[\hbar]]$) to the sections $s_i$'s, we obtain flat sections of $\W_{X,\C}[[\hbar]]$ under the Fedosov connection $D_F$. When the first de Rham cohomology group $H^1_{dR}(X)$ vanishes, these flat sections corresponds to the images of quantum moment maps under the isomorphism $C^\infty(X)[[\hbar]]\cong\Gamma^{flat}(X,\W_{X,\C}[[\hbar]])$.
\end{lem}
\begin{proof}
   Using Lemma \ref{proposition: Lie-derivative-bracket}, we have
	\begin{align*}
		\frac{1}{\hbar}[[s_i,-]_\star, D_F]=\left([\mathcal{L}_i,D_F]-[[D_F,\iota_i], D_F]\right)
	=-\left([[D_F,\iota_i], D_F]\right)=0;
	\end{align*}
	here we used equation \eqref{equation: Fedosov-connection-invariant} in the above lemma in the second equality. On the other hand, there is
	$$
	\frac{1}{\hbar}[[s_i,-]_\star, D_F]=\pm\frac{1}{\hbar}[D_F(s_i),-]_\star=\pm\frac{1}{\hbar}[D_F(s_i),-]_\star=0. 
	$$
	It follows that for every $1\leq i\leq\dim(G)$, we have $D_F(s_i)\in\A^1(X)$  (since $\A^\bullet(X)$ is the center of $\A_X^\bullet(\W_{X,\C})$). Moreover, these $1$-forms must be closed since $D_F^2(s_i)=d_X(D_F(s_i))=0$. The first statement follows because locally we can take anti-derivatives of closed $1$-forms. Globally, vanishing of the first de Rham cohomology group also implies the existence of the anti-derivatives we need, and these are images of quantum moment maps by Proposition \ref{corollary: Lie-derivative-bracket}.
\end{proof}
To summarize, we have
\begin{thm}\label{theorem: quantum-moment-maps-quantizable-functions}
	The images of quantum moment maps are first order quantizable functions.
	Thus, for any $\alpha$ and $k$, there exists a Lie algebra homomorphism 
	$$
	S:\mathfrak{g}\rightarrow C^\infty_{\alpha,k}(X),
	$$
	where the Lie bracket on the right hand side is induced by the star product $\star_k$. 
\end{thm}
\begin{proof}
	We only need to show that every section $s_i$ has a uniformly bounded polynomial degree in $\overline{\W}_X$. This follows from some simple observations on its defining equation \eqref{equation: flat-section-s-i}. First of all, the component in $s_i$ corresponding to the operator $\mathcal{L}_i-[\nabla,\iota_i]$ must be of polynomial degree $2$ and lives in $TX\otimes \overline{TX}$ (since $\mathcal{L}_i$ is a derivation with respect to the classical product on the Weyl bundle). On the other hand, the term $I$ in the Fedosov connection satisfies our desired finiteness property by its construction, and so does the term in $s_i$ corresponding to $[D_F-\nabla,\iota_i]$. Hence, we conclude that $s_i$ has a uniformly bounded polynomial degree in $\overline{\W}_X$.
\end{proof}

\appendix
\renewcommand*{\thesection}{\Alph{section}}

\section{Proof of Proposition \ref{lemma: flat-section-Bargmann-Fock}}\label{appendix: flat-section-Bargmann-Fock}
Notice that the Karabegov form in this situation is $\omega-\alpha=\omega-\hbar\cdot R_{i\bar{j}k}^kdz^i\wedge d\bar{z}^j$. Recall that, after Theorem \ref{theorem: Fedosov-connection}, we write term $I$ in the Fedosov connection as $I=\sum_{i\geq 2}I_i$. Let us write each $I_n$ explicitly as 
$$
I_n=R^j_{i_1\cdots i_n,\bar{l}}\omega_{j\bar{k}}d\bar{z}^l\otimes(y^{i_1}\cdots y^{i_n}\bar{y}^k).
$$
\begin{lem}\label{lemma: trace-I-Ricci-form}
We have
$
(J_\alpha)_n=-(n+1)\hbar\cdot R_{ii_1\cdots i_{n},\bar{l}}^id\bar{z}^l\otimes y^{i_1}\cdots y^{i_n}
$
\end{lem}
\begin{proof}
The proof is by induction on $n$.
For $n=1$, we have
$$
(J_\alpha)_1=(\delta^{1,0})^{-1}\left(-\hbar\cdot R_{i\bar{j}k}^kdz^i\wedge d\bar{z}^j\right)=-2\hbar\cdot\left(\frac{1}{2} R_{i\bar{j}k}^kd\bar{z}^j\otimes y^i\right).
$$
Then, by the induction hypothesis for $n-1$, we have
\begin{align*}
 \nabla^{1,0}(J_\alpha)_{n-1}=\nabla^{1,0}\left(-n\hbar\cdot R_{ii_1\cdots i_{n-1},\bar{l}}^id\bar{z}^l\otimes y^{i_1}\cdots y^{i_{n-1}}\right).
\end{align*}
On the other hand, 
\begin{align*}
 &\nabla^{1,0}\left(n\hbar\cdot R_{i_1\cdots i_{n},\bar{l}}^jd\bar{z}^l\otimes y^{i_1}\cdots y^{i_{n}}\otimes\partial_{y^j}\right)\\
 = & (n+1)\cdot n\hbar\cdot R_{i_1\cdots i_{n+1},\bar{l}}^jdz^{i_{n+1}}\wedge d\bar{z}^l\otimes y^{i_1}\cdots y^{i_n}\otimes\partial_{y^j}.
\end{align*}
Since $\nabla^{1,0}$ is compatible with the contraction between $TX$ and $T^*X$, the above computation shows that
$$
(J_\alpha)_n=(\delta^{1,0})^{-1}(\nabla^{1,0}(J_\alpha)_{n-1})=-(n+1)\hbar\cdot R_{i_1\cdots i_{n+1},\bar{l}}^{i_1} d\bar{z}^l\otimes y^{i_1}\cdots y^{i_n}y^{i_{n+1}}.
$$
\end{proof}

\begin{lem}
 The section $\beta$ satisfies 
 $
 D_{\alpha,k}(\beta)=-\omega_{i\bar{j}}d\bar{z}^j\otimes y^i-\partial\rho.
 $
\end{lem}
\begin{proof}
The function $\rho$ satisfies the condition that $\bar{\partial}\partial(\rho)=\omega$. Recall that $\beta=\sum_{k\geq 1}(\tilde{\nabla}^{1,0})^k(\rho)$, and it is easy to check that $\sigma\left(D_{\alpha,k}(\beta)\right)=\sigma(-\delta(\tilde{\nabla}^{1,0}(\rho))=-\partial\rho$. On the other hand, the following computation shows that $-\omega_{i\bar{j}}d\bar{z}^j\otimes y^i-\partial\rho$ is closed under $D_{\alpha,k}$: 
\begin{align*}
&D_{\alpha,k}(-\omega_{i\bar{j}}d\bar{z}^j\otimes y^i-\partial\rho)\\
=&\nabla(-\omega_{i\bar{j}}d\bar{z}^j\otimes y^i)-\delta(-\omega_{i\bar{j}}d\bar{z}^j\otimes y^i)+k\cdot[I_\alpha,-2\omega_{i\bar{j}}d\bar{z}^j\otimes y^i]_{\star_k}-\bar{\partial}\partial\rho\\
=&\delta(\omega_{i\bar{j}}d\bar{z}^j\otimes y^i)-\bar{\partial}\partial\rho
=\omega_{i\bar{j}}dz^i\wedge d\bar{z}^j-\omega
=0.
\end{align*}
Here we have used the fact that $\omega$ is parallel with respect to $\omega$. Since $\beta$ is a section of the holomorphic Weyl bundle $\W_X$, so is its differential $D_{\alpha,k}(\beta)\in A^1_X(\W_X)$. Furthermore,  we have $D_{\alpha,k}^{1,0}(\beta)=-\rho$, which implies that 
$$
\gamma:=D_\alpha(\beta)+\omega_{i\bar{j}}d\bar{z}^j\otimes y^i+\partial\rho\in\A_X^{0,1}(\W_X).
$$
Suppose $\gamma$ does not vanish. Then $\delta(\gamma)\not=0$ which implies the non-vanishing of $D_{\alpha,k}(\gamma)$. This is a contradiction. 
\end{proof}

\begin{lem}
We have $(I+J_\alpha)\circledast_k(e^{k\cdot\beta})=\left(\sum_{n\geq 2}\tilde{R}_n^*(k\cdot\beta)\right)\circledast_k e^{k\cdot\beta}=k[I,k\cdot\beta]_\star\circledast_k e^{k\cdot\beta}$.
\end{lem}
\begin{proof}
For every $n\geq 2$, there is the following straightforward computation:
\begin{align*}
 &k\cdot I_n\circledast_k (e^{k\cdot\beta}\otimes e_{L^{k}})\\
 =&-2\sqrt{-1}\cdot R_{i_1\cdots i_n,\bar{l}}^j\omega_{j\bar{k}}d\bar{z}^l\otimes(y^{i_1}\cdots y^{i_n}\bar{y}^k)\circledast_k(e^{k\cdot\beta}\otimes e_{L^{k}})\\
 =&-2\sqrt{-1}\cdot R_{i_1\cdots i_n,\bar{l}}^j\omega_{j\bar{k}}d\bar{z}^l\otimes(\frac{\omega^{i\bar{k}}}{2\sqrt{-1}}\frac{\partial}{\partial y^i})(y^{i_1}\cdots y^{i_n}e^{k\cdot\beta}\otimes e_{L^{k}})\\
 =&R_{i_1\cdots i_n,\bar{l}}^id\bar{z}^l\otimes y^1\cdots y^n\frac{\partial(k\cdot\beta))}{\partial y^i}\cdot(e^{k\cdot\beta}\otimes e_{L^{k}})+n\cdot R_{ii_1\cdots i_{n-1},\bar{l}}^id\bar{z}^l\otimes y^{i_1}\cdots y^{i_{n-1}}\cdot(e^{k\cdot\beta}\otimes e_{L^{k}})\\
 =&\left(\tilde{R}_n^*(k\cdot\beta)+n\cdot R_{ii_1\cdots i_{n-1},\bar{l}}^id\bar{z}^l\otimes y^{i_1}\cdots y^{i_{n-1}}\right)\cdot(e^{k\cdot\beta}\otimes e_{L^{k}})\\
 =&\left(\tilde{R}_n^*(k\cdot\beta)-(J_\alpha)_{n-1}\right)\circledast_k(e^{k\cdot\beta}\otimes e_{L^{k}}).
\end{align*}
\end{proof}
Summarizing the above computations, we have
\begin{align*}
 &D_{\alpha,k}(e^{k\cdot\beta}\otimes e_{L^{k}})\\
 =&\left(\nabla+k\cdot\gamma_\alpha\circledast_k\right)(e^{k\cdot\beta}\otimes e_{L^{ k}})+e^{k\cdot\beta}\otimes \nabla_{L^{\otimes k}}(e_{L^{k}})\\
 =&\left(\nabla(k\cdot\beta)+k\cdot\omega_{i\bar{j}}(d\bar{z}^j\otimes y^i-dz^i\otimes \bar{y}^j)\circledast_k+k(I+J_\alpha)\circledast_k\right)(e^{k\cdot\beta}\otimes e_{L^{k}})+e^{k\cdot\beta}\otimes (k\partial{\rho}\cdot e_{L^{k}})\\
 =&\left(\nabla(k\cdot\beta)+k\cdot\omega_{i\bar{j}}d\bar{z}^j\otimes y^i+k\cdot\partial\rho\right)(e^{k\cdot\beta}\otimes e_{L^{k}})\\
 &+k(-\omega_{i\bar{j}}dz^i\otimes\bar{y}^j+I+J_\alpha)\circledast_k(e^{k\cdot\beta}\otimes e_{L^{k}})\\
 =&\left(\nabla(k\cdot\beta)+k\cdot\omega_{i\bar{j}}d\bar{z}^j\otimes y^i+k\cdot\partial\rho+(-\omega_{i\bar{j}})dz^i(-\omega^{k\bar{j}})\frac{\partial(k\cdot\beta)}{\partial y^k}+k[I,k\cdot\beta]_{\star_k}\right)\\
 &\ \circledast_k(e^{k\cdot\beta}\otimes e_{L^{k}})\\
  =&\left(\nabla(k\cdot\beta)+k[I,k\cdot\beta]_{\star_k}+k\cdot\omega_{i\bar{j}}d\bar{z}^j\otimes y^i+k\cdot\partial\rho-\delta^{1,0}(k\cdot\beta)\right)\circledast_k(e^{k\cdot\beta}\otimes e_{L^{ k}})\\
  =&k\cdot\left(D_{\alpha,k}(\beta)+\omega_{i\bar{j}}d\bar{z}^j\otimes y^i+\partial\rho\right)\circledast_k(e^{k\cdot\beta}\otimes e_{L^{ k}})\\
  =&0.
\end{align*}
This completes the proof of Proposition \ref{lemma: flat-section-Bargmann-Fock}.


\begin{bibdiv}
\begin{biblist}
	

\bib{Witten}{article}{
	AUTHOR = {Axelrod, S.},
	author = {Della Pietra, S.},
	author = {Witten, E.},
	TITLE = {Geometric quantization of {C}hern-{S}imons gauge theory},
	JOURNAL = {J. Differential Geom.},
	VOLUME = {33},
	YEAR = {1991},
	NUMBER = {3},
	PAGES = {787--902},
}



\bib{Bischoff-Gualtieri}{article}{
	AUTHOR = {Bischoff, F.},
	AUTHOR = {Gualtieri, M.},
	TITLE = {Brane quantization of toric {P}oisson varieties},
	JOURNAL = {Comm. Math. Phys.},
	FJOURNAL = {Communications in Mathematical Physics},
	VOLUME = {391},
	YEAR = {2022},
	NUMBER = {2},
	PAGES = {357--400},
}

\bib{Bordemann-Meinrenken}{article}{
	AUTHOR = {Bordemann, M.},
	author = {Meinrenken, E.},
	author = {Schlichenmaier, M.},
	TITLE = {Toeplitz quantization of {K}\"{a}hler manifolds and {${\rm
				gl}(N)$}, {$N\to\infty$} limits},
	JOURNAL = {Comm. Math. Phys.},
	VOLUME = {165},
	YEAR = {1994},
	NUMBER = {2},
	PAGES = {281--296},
}

\bib{Bordemann}{article}{
    AUTHOR = {Bordemann, M.},
    author = {Waldmann, S.},
     TITLE = {A {F}edosov star product of the {W}ick type for {K}\"{a}hler
              manifolds},
   JOURNAL = {Lett. Math. Phys.},
    VOLUME = {41},
      YEAR = {1997},
    NUMBER = {3},
     PAGES = {243--253},
}

\bib{Bott}{article}{
	AUTHOR = {Bott, R.},
	TITLE = {Homogeneous vector bundles},
	JOURNAL = {Ann. of Math. (2)},
	VOLUME = {66},
	YEAR = {1957},
	PAGES = {203--248},
}

\bib{CLL3}{article}{
	author={Chan, K.},
	author={Leung, N. C.},
	author={Li, Q.},
	TITLE = {Bargmann-{F}ock sheaves on {K}\"{a}hler manifolds},
	JOURNAL = {Comm. Math. Phys.},
	VOLUME = {388},
	YEAR = {2021},
	NUMBER = {3},
	PAGES = {1297--1322},
}

\bib{CLL}{article}{
   author={Chan, K.},
   author={Leung, N. C.},
   author={Li, Q.},
 TITLE = {Kapranov's {$L_\infty$} structures, {F}edosov's star products,
	and one-loop exact {BV} quantizations on {K}\"{a}hler manifolds},
JOURNAL = {Commun. Number Theory Phys.},
VOLUME = {16},
YEAR = {2022},
NUMBER = {2},
PAGES = {299--351},
}

\bib{EF}{article}{
	author={Engeli, M.},
	author={Felder, G.},
	TITLE={A Riemann-Roch-Hirzebruch formula for traces of differential operators},
	JOURNAL={Ann. Scient. Éc. Norm. Sup.},
	VOLUME={41},
	YEAR={2008},
	NUMBER={4},
	PAGES={623--655},
}	
	
\bib{Fed}{article}{
    AUTHOR = {Fedosov, B. V.},
     TITLE = {A simple geometrical construction of deformation quantization},
   JOURNAL = {J. Differential Geom.},
    VOLUME = {40},
      YEAR = {1994},
    NUMBER = {2},
     PAGES = {213--238}
}



\bib{Ginzburg}{article}{
	AUTHOR = {Ginzburg, V.},
	TITLE = {Lectures on {$\mathcal{D}$}-modules},
	NOTE = {Available online},
	YEAR = {1998 Chicago notes},
}

\bib{GLL}{article}{
	AUTHOR = {Grady, R.},
	author = {Li, Q.},
	author = {Li, S.},
	TITLE = {Batalin-{V}ilkovisky quantization and the algebraic index},
	JOURNAL = {Adv. Math.},
	VOLUME = {317},
	YEAR = {2017},
	PAGES = {575--639},
}


\bib{GW}{article}{
    AUTHOR = {Gukov, S.},
    author = {Witten, E.},
     TITLE = {Branes and quantization},
   JOURNAL = {Adv. Theor. Math. Phys.},
    VOLUME = {13},
      YEAR = {2009},
    NUMBER = {5},
     PAGES = {1445--1518},
}

		
\bib{Kapranov}{article}{
	AUTHOR = {Kapranov, M.},
	TITLE = {Rozansky-{W}itten invariants via {A}tiyah classes},
	JOURNAL = {Compositio Math.},
	VOLUME = {115},
	YEAR = {1999},
	NUMBER = {1},
	PAGES = {71--113},
}		


\bib{Karabegov96}{article}{
    AUTHOR = {Karabegov, A.V.},
     TITLE = {Deformation quantizations with separation of variables on a
              {K}\"{a}hler manifold},
   JOURNAL = {Comm. Math. Phys.},
    VOLUME = {180},
      YEAR = {1996},
    NUMBER = {3},
     PAGES = {745--755},
}

\bib{Karabegov00}{incollection}{
    AUTHOR = {Karabegov, A.V.},
     TITLE = {On {F}edosov's approach to deformation quantization with
              separation of variables},
 BOOKTITLE = {Conf\'{e}rence {M}osh\'{e} {F}lato 1999, {V}ol. {II} ({D}ijon)},
    SERIES = {Math. Phys. Stud.},
    VOLUME = {22},
     PAGES = {167--176},
 PUBLISHER = {Kluwer Acad. Publ., Dordrecht},
      YEAR = {2000},
}

\bib{Karabegov07}{article}{
    AUTHOR = {Karabegov, A.V.},
     TITLE = {A formal model of {B}erezin-{T}oeplitz quantization},
   JOURNAL = {Comm. Math. Phys.},
    VOLUME = {274},
      YEAR = {2007},
    NUMBER = {3},
     PAGES = {659--689},
}

\bib{Karabegov}{article}{
    AUTHOR = {Karabegov, A.V.},
    author = {Schlichenmaier, M.},
     TITLE = {Identification of {B}erezin-{T}oeplitz deformation
              quantization},
   JOURNAL = {J. Reine Angew. Math.},
    VOLUME = {540},
      YEAR = {2001},
     PAGES = {49--76},
}

\bib{Lerman}{incollection}{
	AUTHOR = {Lerman, E.},
	TITLE = {Geometric quantization; a crash course},
	BOOKTITLE = {Mathematical aspects of quantization},
	SERIES = {Contemp. Math.},
	VOLUME = {583},
	PAGES = {147--174},
	PUBLISHER = {Amer. Math. Soc., Providence, RI},
	YEAR = {2012},
}


\bib{Ma-Ma-1}{article}{
    AUTHOR = {Ma, X.},
    author = {Marinescu, G.},
     TITLE = {Toeplitz operators on symplectic manifolds},
   JOURNAL = {J. Geom. Anal.},
    VOLUME = {18},
      YEAR = {2008},
    NUMBER = {2},
     PAGES = {565--611},
}

\bib{Ma-Ma}{article}{
    AUTHOR = {Ma, X.},
    author = {Marinescu, G.},
     TITLE = {Berezin-{T}oeplitz quantization on {K}\"{a}hler manifolds},
   JOURNAL = {J. Reine Angew. Math.},
    VOLUME = {662},
      YEAR = {2012},
     PAGES = {1--56},

}

\bib{Neumaier}{article}{
    AUTHOR = {Neumaier, N.},
     TITLE = {Universality of {F}edosov's construction for star products of
              {W}ick type on pseudo-{K}\"{a}hler manifolds},
   JOURNAL = {Rep. Math. Phys.},
    VOLUME = {52},
      YEAR = {2003},
    NUMBER = {1},
     PAGES = {43--80},
}


\bib{Xu}{article}{
	AUTHOR = {Xu, P.},
	TITLE = {Fedosov {$*$}-products and quantum momentum maps},
	JOURNAL = {Comm. Math. Phys.},
	VOLUME = {197},
	YEAR = {1998},
	NUMBER = {1},
	PAGES = {167--197},
}

		



\end{biblist}
\end{bibdiv}
\end{document}